\documentclass[10pt,oneside,reqno]{amsart}

\usepackage[utf8]{inputenc}
\usepackage[T1]{fontenc}
\usepackage[ngerman,english]{babel}
\usepackage{lmodern}

\usepackage{hyperref}

\usepackage{amsmath}
\usepackage{amsthm}
\usepackage{amsfonts}
\usepackage{amssymb}
\usepackage{amscd}
\usepackage{amsbsy}
\usepackage{bm}
\usepackage{csquotes}

\usepackage{pdfpages}
\usepackage{fancyhdr}
\usepackage{geometry}
\usepackage{setspace}
\usepackage{xcolor}
\usepackage{pifont}

\usepackage{graphicx}
\usepackage{tabularx}
\usepackage{epsfig}
\usepackage[all]{xy}
\usepackage{tikz}
\usetikzlibrary{matrix}

\usepackage{fixmath}

\usepackage{appendix}
\usepackage[sort, numbers]{natbib}
\usepackage{bibentry}

\usepackage{enumerate}

\newtheorem{theorem}{Theorem}[section]
\newtheorem{lemma}[theorem]{Lemma}

\newtheorem{corollary}[theorem]{Corollary}

\theoremstyle{definition}
\newtheorem{definition}[theorem]{Definition}

\newtheorem*{remark}{Remark}

\numberwithin{equation}{section}

\newcommand{\N}{{\mathbb N}}
\newcommand{\Z}{{\mathbb Z}}

\newcommand{\R}{{\mathbb R}}
\newcommand{\C}{{\mathbb C}}
\newcommand{\D}{{\mathbb D}}

\renewcommand{\Re}{\operatorname{Re}}
\renewcommand{\Im}{\operatorname{Im}}

\newcommand{\inv}{^{-1}}

\newcommand{\xlim}{\underset{x \rightarrow \infty}{\lim}}
\newcommand{\smat}{\begin{pmatrix}}
\newcommand{\fmat}{\end{pmatrix}}

\newcommand{\vertiii}[1]{{\left\vert\kern-0.25ex\left\vert\kern-0.25ex\left\vert #1 
    \right\vert\kern-0.25ex\right\vert\kern-0.25ex\right\vert}}

  \DeclareMathOperator\Var{Var}

\title{Dirac operators with operator data of Wigner-von Neumann type}

\thanks{E.G.\ was supported in part by NSF grant DMS--1745670.}

\author{Ethan~Gwaltney}
\begin{document}

\address{Department of Mathematics, Rice University MS-136, Box 1892,
Houston, TX 77251-1892, USA.}
\email{ethan.gwaltney@rice.edu}

\begin{abstract}
We consider half-line Dirac operators with operator data of Wigner-von Neumann type. 
If the data is a finite linear combination of Wigner-von Neumann functions, we show absence of singular continuous spectrum and provide an explicit set containing all embedded pure points depending only on the $L^p$ decay and frequencies of the operator data. 
For infinite sums of Wigner-von Neumann-like terms, we bound the Hausdorff dimension of the singular part of the spectrum. 
\end{abstract} 
\maketitle

\begin{section}{Introduction}
In 1929, von Neumann and Wigner \cite{vNW29} constructed a potential of a one-dimensional continuum Schr\"odinger operator behaving at infinity as 
\[
	V(x) = -8\frac{\sin(2x)}{x} + O(x^{-2}),
\]
for which $E=1$ is an eigenvalue embedded in the absolutely continuous part of the spectrum.
Many variants of this model have since been used to demonstrate and study `exotic' singular spectrum (e.g., \cite{A54,BD79,FL21,HL75,K2012,LO2015,NS12,S12,W83}). 
One construction due to Simon uses the Wigner-von Neumann model as the basic building block in a potential for which the associated Schr\"odinger operator exhibits dense embedded point spectrum \cite{S97}. Two results of Simonov even precisely describes the asymptotics of the spectral density near a critical point for certain Wigner-von Neumann-like perturbations of a periodic potential \cite{Sim2016}. 

Functions of generalized bounded variation, of which von Neumann and Wigner's potential is a special case, combine slower decay with additional Wigner-von Neumann-like terms with differing frequencies, producing a mixture of bounded variation, decay at infinity, and almost periodicity \cite{W09}. 
This combination is an interesting one for at least three reasons: potentials of bounded variation with decay at infinity preserve the absolutely continuous spectrum \cite{W67}; $L^1$ potentials preserve the purely absolutely continuous spectrum (see, for example, \cite{S96}), which underscores the importance of the rate of decay for producing an embedded eigenvalue; and there is a generic set of almost periodic potentials producing purely singular spectrum \cite{A2009}. 

Historically, most exploration of exotic spectra arising from models based on the Wigner-von Neumann potential restricted to the $L^2$ case, with one result \cite{JS10} from the discrete case allowing $\ell^3$ decay.  
Luki\'c \cite{L2014,L2013DiscreteGBV,L2011,L2013SlowWvN} used functions of generalized bounded variation to progress to the $L^p$ setting for any integer $2 \leq p < \infty$, showing the absence of singular continuous spectrum and explicitly providing $p$-dependent finite sets containing all possible instances of embedded pure points.
In \cite{L2014}, Luki\'c extended this work to include potentials with infinitely many summands of generalized bounded variation, in which case the set of possible pure points is in general infinite.

Schr\"odinger operators and Dirac operators have often been studied in tandem. 
For example, Behncke \cite{B91} established criteria for subordinate solutions in both the Schr\"odinger and Dirac settings, and Naboko \cite{N86} demonstrated dense point spectrum in the absolutely continuous spectrum of Dirac operators and deduced the same for Schr\"odinger operators as a special case. 
While there are some considerations of Dirac operators with Wigner-von Neumann type operator data in the literature (e.g., \cite{BR95,LS2017}), such considerations seem to be rare. 
Here we continue the work of Luki\'c on spectral type charcterization of models with Wigner-von Neumann type data in the context of the half-line Dirac operator. 
The analysis begins in much the same way as in \cite{L2013SlowWvN}, but an important adaptation is required that alters the analysis throughout and the results we obtain. 

The Dirac operator $L_\varphi$ is defined by the expression
\begin{equation}
	L_\varphi = \begin{pmatrix} 0 & -1 \\ 1 & 0 \end{pmatrix} \partial_x + \begin{pmatrix} R(x) & I(x) \\ I(x) & -R(x) \end{pmatrix}, \label{dirac.gauge}
\end{equation}
where the real-valued functions $R(x)$ and $I(x)$ should be viewed as the real and imaginary parts of a complex-valued function $\varphi(x).$ 
The function $\varphi$ arises naturally in the Zakharov-Shabat operator
\begin{equation}
	\Lambda_\varphi = \begin{pmatrix} i & 0 \\ 0 & -i \end{pmatrix} \frac{d}{dx} + \begin{pmatrix} 0 & \varphi(x) \\ \overline{\varphi(x)} & 0 \end{pmatrix}, \label{ZSO}
\end{equation}
which differential expression is unitarily equivalent to \eqref{dirac.gauge}. 
The latter form is often more convenient for calculations, and below we opt to work with $\Lambda_\varphi$. 
More details regarding this gauge distinction are available in \cite{CG2002,EGL2020,GK2014}. 
We will study the eigenequation for $\Lambda_\varphi$, 
\begin{equation}
	\Lambda_\varphi U(x,\eta) = E U(x,\eta), \label{eigenequation}
\end{equation}
where $\eta = 2E.$ 
Our operator data $\varphi$ will be of almost the same form as the potentials in \cite{L2014}, which form generalizes the famous Wigner-von Neumann potential of \cite{vNW29}. 
The only adjustment to the definition from \cite{L2014} is that, for reasons we will explain shortly, we may without loss of generality restrict to odd integers $p$.
We recall that the variation of a function $\gamma$ on an interval $I$ is defined as
\[
	\Var(\gamma, I) = \sup_{k \in \N} \sup_{\substack{x_0, \ldots, x_k \in I \\ x_0 < \cdots < x_k}} \sum_{j=1}^{k} |\gamma(x_j) - \gamma(x_{j-1})|.
\]
\begin{definition}
\label{WVNtype.defn}
We say $\varphi$ is of \textit{Wigner-von Neumann type} if it takes the form
\begin{align}
\label{WvNPotential}
	\varphi(x) = \sum_{j =1}^\infty c_j e^{-i \phi_j x} \gamma_j(x),
\end{align}
where $c_j \in \C$, $\phi_j \in \R$, and all of the following conditions hold:
\begin{enumerate}
	\item (uniformly bounded variation) the functions $\gamma_j:(0,\infty) \to \C$ obey
	\begin{equation}
		\sup_{j} \text{Var}(\gamma_j, (0,\infty)) < \infty. \label{uniformly.b.v.condition}
	\end{equation}
	\item (uniform $L^p$ condition) for some odd $p \in \Z$, $p \geq 3$, 
	\begin{equation}
		\sup_{j } \int_0^\infty |\gamma_j(t)|^p dt < \infty. \label{uniformly.L.p.condition}
	\end{equation}
	\item ($\alpha$-type decay of coefficients) for some $\alpha \in (0, \frac{1}{p-2})$, 
	\begin{equation}
		\sum_{j}|c_j|^\alpha < \infty. \label{alpha.type.decay.condition}
	\end{equation}
\end{enumerate}
If only finitely many $c_j$ are nonzero, we say $\varphi$ is of \textit{finite Wigner-von Neumann type}. 
In this latter case, of course, the condition \eqref{alpha.type.decay.condition} holds for arbitrary positive $\alpha$. 
We call the $\phi_j$ \textit{frequencies} and denote the set of all frequencies by $\Phi = \{\phi_j: j \in \N\}$. 
\end{definition}
The differential expression \eqref{ZSO} with operator data $\varphi$ of Wigner-von Neumann type has $0$ as a regular endpoint and, since $\varphi$ decays at infinity, is in the limit point case at $+\infty$. 
Thus, $\Lambda_\varphi^\omega$ defines an unbounded self-adjoint operator with domain
\[
	D(\Lambda_\varphi^\omega) = \{f \in H^1( (0,\infty), \C^2): \begin{pmatrix}\omega & \overline{\omega} \end{pmatrix} f(0) = 0\},
\]
where we have parametrized $\Lambda_\varphi^\omega$ by a unimodular complex constant $\omega$.
The choice of $\omega \in \partial \D$ is not important to the analysis--our results stand as stated independent of the choice of boundary condition. 

The first of our main results shows absence of singular continuous spectrum and provides an explicit set containing all possible embedded pure points in the case where $\varphi$ is of finite Wigner-von Neumann type:
\begin{theorem}
\label{finite.main.thm}
	Let $\Lambda_\varphi$ be a Dirac operator with operator data $\varphi$ of finite Wigner-von Neumann type satisfying the uniform $L^p$ condition \eqref{uniformly.L.p.condition} for some $p = 2n+1$, $n \geq 1$. 
	Then for 
	\[
		S_p = \Big\{\frac{\eta}{2}\big\vert \eta = \sum_{j=1}^{m}\phi_{k_j} - \sum_{j=1}^{m-1}\phi_{l_j}, \phi_{k_j}, \phi_{l_j} \in \Phi, 1 \leq m \leq n\Big\},
	\]	
	which depends only on $p$ and the set $\Phi$ of frequencies of $\varphi$, the spectral measure $\mu$ of $L_\varphi$ is mutually absolutely continuous with Lebesgue measure on $\R \setminus S_p$. 
	Consequently, 
	\begin{enumerate}
		\item $\sigma_{ac}(L_\varphi) = \R$
		\item $\sigma_{sc}(L_\varphi) = \emptyset$
		\item $\sigma_{pp}(L_\varphi) \subset S_p$ is a finite set.
	\end{enumerate} 
\end{theorem}
In addition to technical adjustments to the methods in \cite{L2011,L2013SlowWvN}, we encounter an exceptional set similar to that in the setting of orthogonal polynomials on the unit circle, rather than the exceptional sets in the settings of Schr\"odinger operators or orthogonal polynomials on the real line. 
Namely, not all sums and differences of frequencies from the operator data give rise to possible pure points, but rather only those of the form
\begin{equation}
	\sum_{j=1}^{m} \phi_{k_j} - \sum_{j=1}^{m-1} \phi_{l_j}. \label{form.of.singularities}
\end{equation}
The intuitive reason for this is reminiscent of the setting of orthogonal polynomials on the unit circle. 
On the unit circle, rotating the measure by an angle $\psi$ shifts each of the frequencies $\phi_j$ by $\psi$. 
Thus, from the set of \textit{a priori} possible critical points, only those of the form \eqref{form.of.singularities} are preserved. 
Similarly, to shift the spectral paramater $E = \eta/2$ by $\psi$ in the Dirac setting, we multiply our operator data $\varphi$ by $e^{i2\psi x}$, which shifts each frequency $\phi_j$ by $2\psi$, and again we see that only critical points $\eta$ of the form \eqref{form.of.singularities} are preserved. 
The fact that new elements of the form \eqref{form.of.singularities} become available only when $p$ increases to an odd integer accounts for the odd $p$ in the $L^p$ condition of the theorem. 
Our insistence that $p$ be odd in Definition \ref{WVNtype.defn} is only a choice made for convenience and not a substantial change to the definition in \cite{L2014}. 

Since $S_p$ grows as $n$ grows in $p=2n+1$, it is natural to ask both whether there exist $\varphi$ for which $S_p$ indeed contains an embedded eigenvalue and whether the growth in the sets $S_p$ is necessary or an artifact of our method. 
To answer these questions, we construct operator data that, for the right choice of boundary condition, yields an eigenvalue in $S_5 \setminus S_{3}$. 
A similar argument is available to produce $\varphi$ with eigenvalue $E \in S_p \setminus S_{p-2}$ for larger $p$. 
Our construction will use operator data of the form
\begin{equation}
	\varphi(x) = \sum_{j=1}^{M} c_j x^{-\delta} e^{-i(\phi_j x + \xi_j(x))}, \label{special.operator.data.form}
\end{equation}
where $\delta \in (p\inv, (p-2)\inv]$, $\xi_j$ are real-valued, and $M < \infty$. 
We can realize $\gamma_j(x)$ in \eqref{WvNPotential} as $c_jx^{-\delta} e^{-i\xi_j(x)}$. 
Thus defined, $\varphi$ satisfies the conditions of Definition \ref{WVNtype.defn}. 

\begin{theorem}
\label{example.thm}
Let $\varphi$ be given by \eqref{special.operator.data.form} with $p =5$, $M \geq 3,$ $\delta \in (p\inv, (p-2)\inv]$, and the frequencies $\phi_j$ rationally independent. 
Then there exists $E=\eta/2 \in S_p \setminus S_{p-2}$. 
Moreover, for any $c_j$ and $\phi_j$ satisfying both 
\begin{equation}
	\sum_{j=1}^{M} \frac{|c_j|^2}{\phi_j - \eta} = 0 \label{second.order.condition}
\end{equation}
and 
\begin{equation}
	\sum_{j_1,j_2=1}^{M} |c_{j_1}c_{j_2}|^2 \frac{\phi_{j_1} + \phi_{j_2} - 2\eta}{(\phi_{j_1}- \eta)^2(\phi_{j_2}-\eta)^2} = 0, \label{fourth.order.condition} 
\end{equation}
there exist functions $\xi_j \in C^1$ such that $\varphi$ satisfies \eqref{uniformly.b.v.condition} and \eqref{eigenequation} has a solution with asymptotics \eqref{asymptotics.one} and \eqref{asymptotics.two}. 
In particular, $\sigma_{\text{ac}}(\Lambda_\varphi)$ has an embedded pure point at $E$. 
\end{theorem}
We will also show that both conditions \eqref{second.order.condition} and \eqref{fourth.order.condition} may be simultaneously satisfied. 

If $\varphi$ is of Wigner-von Neumann type with infinitely many nonzero $c_j$ and $p \geq 3$, each of the frequencies $\phi_j$ is of the form \eqref{form.of.singularities} with $n=1$, so that the set $S_p$ is in general infinite. 
In this case we instead bound the Hausdorff dimension of the singular part of the spectral measure. 
This requires careful bookkeeping of the critical points in $\eta$. 
We will summarize the critical point information via the following recursively defined rational functions: let $h_I(\eta; [\phi_{k_j}]_{j=1}^P; [\phi_{l_j}]_{j=1}^{P-1})$ be defined for $I=2P-1$ by $h_1(\eta;[\phi_{k_1}]) = (\phi_{k_1}-\eta)\inv$ and 
\begin{align}
	h_I(\eta; [\phi_{k_j}]_{j=1}^P; [\phi_{l_j}]_{j=1}^{P-1}) &= \notag \\
	&\hspace{-.5in}\frac{1}{\sum_{j=1}^{P}\phi_{k_j} - \sum_{j=1}^{P-1}\phi_{l_j}-\eta} \sum_{m=0}^{I-1}h_m(\eta; [\phi_{k_j}]_{j=1}^{\frac{m+1}{2}}; [\phi_{l_j}]_{j=1}^{\frac{m-1}{2}})h_{I-1-m}(\eta; [\phi_{k_j}]_{j=\frac{m+3}{2}}^P; [\phi_{l_j}]_{j=\frac{m+1}{2}}^{P-2}), \label{h.I}
\end{align}
where we denote by $[\phi_j]_{j=1}^n$ an ordered $n$-tuple of frequencies. 
For even $I$, we define $h_I\equiv 0$. 
The functions $h_I$ are analogous to the $h_J$ of \cite{L2014}, with additional structure added in order to omit sums of frequencies not of the form \eqref{form.of.singularities}. 
Since the first three conditions of the following lemma are satisfied by our operator data of Wigner-von Neumann type by definition, the following lemma gives a sufficient condition in terms of the functions $h_I$ for boundedness of all solutions at $E$:
\begin{lemma}
\label{crux.lemma}
	Let operator data $\varphi$ be given by \eqref{WvNPotential}, with $c_j \in \C$, $\phi_j \in \R$, and let $\eta \in \R$ such that
	\begin{enumerate}
		\item $\sup_j \Var(\gamma_j, (0,\infty)) := \tau < \infty$;  
		\item for some $p= 2n+1$, $\sup_j \int_{0}^{\infty} |\gamma_j(t)|^p dt := \sigma < \infty$;
		\item $\sum_{j=1}^{\infty} |c_j| < \infty$;
		\item for odd $I = 1, \ldots, p-2$ and $P:= \frac{I+1}{2}$, $$\sum_{\substack{k_1, \ldots, k_P = 1 \\ l_1, \ldots, l_{P-1} = 1}}^{\infty} |h_I(\eta; [\phi_{k_j}]_{j=1}^P; [\phi_{l_j}]_{j=1}^{P-1}) \prod_{j=1}^{P}c_{k_j}\prod_{j=1}^{N}\overline{c_{l_j}}| < \infty.$$ \label{small.divisors.condition}
	\end{enumerate}
	Then, for $E = \frac{\eta}{2}$, all solutions of \eqref{eigenequation} are bounded. 
\end{lemma}
With this lemma in hand, we will be able to bound the Hausdorff dimension of the exceptional set, arriving at our second main result:
\begin{theorem}
\label{infinite.main.thm}
Let operator data $\varphi$ be of Wigner-von Neumann type satisfying the $L^p$ condition for some odd $p$.
Then the set of energies $E$ for which there exists an unbounded solution to \eqref{eigenequation} has Hausdorff dimension at most $(p-2)\alpha$, and $\R$ is the essential support of the absolutely continuous spectrum of $\Lambda_\varphi$. 
\end{theorem}

In Section~\ref{SectionSubordinacyAndPruferVariables}, we define the Pr\"ufer variables to be used throughout. In Section~\ref{SectionNonremovableSingularities} we prove the form \eqref{form.of.singularities} of critical points. In Section~\ref{SectionFinitelyManySummands} we prove Theorems \ref{finite.main.thm} and \ref{example.thm}. In Section~\ref{SectionInfinitelyManySummands} we prove Theorem \ref{infinite.main.thm}. 
\end{section}

\begin{section}{Subordinacy and Pr\"ufer Variables} \label{SectionSubordinacyAndPruferVariables}
A subordinate solution at $E = \eta/2$ is a solution $U(x,\eta)$ to \eqref{eigenequation} such that for any linearly independent solution at $E$, $V(x,\eta)$, we have
\[
	\xlim \frac{\int_{0}^{x}\|U(t,\eta)\|^2 dt}{\int_{0}^{x}\|V(t,\eta)\|^2dt} = 0.
\]
For more on subordinacy theory, we refer the reader to \cite{GP87}. 
By work of Behncke \cite{B91}, for both Schr\"odinger and Dirac operators the nonexistence of subordinate solutions on a set of energies $E$ is sufficient to conclude purely absolutely continuous spectrum there. 
We will find an explicit set of energies $E$ for which all solutions are bounded. 
In the context of Schr\"odinger operators, work of Stolz \cite{S92} shows boundedness of all solutions is sufficient to conclude the nonexistence of subordinate solutions. 
In the same way, the following lemma, following ideas from \cite{S92}, completes the desired chain of implications in the Dirac operator setting. 
\begin{lemma}
\label{stolz.type.lemma}
If all solutions of \eqref{eigenequation} at $E$ are bounded, then there is no subordinate solution at $E$. 
\end{lemma}
\begin{proof}
Let $U$ and $V$ be two linearly independent solutions at $E$. 
Since both are bounded, we can denote the (finite) suprema of their norms by $M_U$ and $M_V$, respectively. 
Clearly, for $x > 0$ we have
\[
	\int_{0}^{x} \|V(t)\|^2 dt \leq M_V^2 x.
\] 
The Wronskian of any $f,g \in H^1( (0,\infty); \C^2)$, $W[f,g](x),$ is defined as $f(x)^tJg(x)$ for $J = \begin{pmatrix} 0 & i \\ -i & 0 \end{pmatrix}$. It is straightforward to show $W[U,V](x) \equiv W$ is a nonzero constant. 
Thus, for any $x$ we have
\[
	|W| = |U(x)^t J V(x)| \leq \|U(x)\|\|JV(x)\| \leq \|U(x)\|M_V.
\] 
Consequently, for any $x>0$,
\[
	\frac{\int_{0}^{x} \|U(t)\|^2 dt}{\int_{0}^{x} \|V(t)\|^2 dt} \geq \frac{|W|^2}{M_V^4} > 0.
\]
Since $U, V$ were chosen arbitrarily, taking $x \to \infty$ shows there is no subordinate solution at $E$. 
\end{proof}
In order to prove boundedness of solutions, we perform a Pr\"ufer transformation to an arbitrary solution $U(x,\eta)$. 
Pr\"ufer variables have been used many times in the spectral theory of Schr\"odinger and Dirac operators (e.g., \cite{K2005,KLS98,L2014,L2013DiscreteGBV,L2011,L2013SlowWvN,LO16,LO2015,P26,S2016,ST2010}).
We set 
\[
	E = \frac{\eta}{2},
\]
and for a solution $U(x)$ of \eqref{eigenequation} at $E$, we define the Pr\"ufer amplitude $r$ and Pr\"ufer angle $\theta$ by
\begin{equation}
	U(x,\eta) = r(x,\eta) \binom{(1+i)e^{-i(\frac{\eta}{2}x + \theta(x,\eta))}}{(1-i)e^{i(\frac{\eta}{2}x + \theta(x,\eta))}}. \label{prufer.variables.defn}
\end{equation}
The ambiguity in $\theta$ is addressed by fixing $\theta(0,\eta) \in [-\pi,\pi)$ and demanding $\theta$ be continuous in $x$. 
So defined, the variables $r$ and $\theta$ satisfy the following system of differential equations (recall that $R = \Re \varphi$ and $I = \Im \varphi$): 
\begin{align*}
	-i\partial_x\theta &= iR(x) \sin(\eta x+ 2\theta(x,\eta)) + iI(x)\cos(\eta x + 2\theta(x,\eta)), \\
	\partial_x\log r &= R(x)\cos(\eta x + 2\theta(x,\eta)) - I(x)\sin(\eta x + 2\theta(x,\eta)).
\end{align*}
Defining the complex Pr\"ufer variable $Z(x,\eta) := r(x,\eta)e^{-i\theta(x,\eta)}$, we have
\[
	\frac{\partial_x Z(x,\eta)}{Z(x,\eta)} = \partial_x\log r(x,\eta) - i\partial_x\theta(x,\eta) = e^{i(\eta x + 2\theta(x,\eta))} \varphi(x).
\]
Thus, $\partial_x\log r(x,\eta) = \Re(e^{i(\eta x + 2\theta(x,\eta))} \varphi(x))$, and
\begin{equation}
	\log \frac{r(x,\eta)}{r(0,\eta)} = \Re \int_{0}^{x} e^{i(\eta t + 2\theta(t,\eta))} \varphi(t) dt. \label{log.r}
\end{equation}
To prove boundedness of solutions at $\eta$, it suffices to bound \eqref{log.r}.
Moreover, we have
\begin{equation}
	\partial_x\theta(x,\eta) = -\Im e^{i(\eta x + 2\theta(x,\eta))} \varphi(x), \label{theta.prime}
\end{equation}
which identity will prove useful in many of our calculations below due to Lemma \ref{exchange.lemma}. 
In later sections, we will often suppress the $\eta$- and $x$-dependence of $r$ and $\theta$ for conciseness. 
\end{section}

\begin{section}{Nonremovable singularities} \label{SectionNonremovableSingularities}
Consider the following reindexing of \cite[Lemma 2.1]{L2014}:
\begin{lemma}
\label{exchange.lemma}
Let $\eta \in \R$ be fixed. Let $P,N\in \Z$ with $P \geq 1$, $I = P+N$, and $K = P-N \geq 0$. 
Then for $0 \leq a < b < \infty$, $\Gamma(x) = \gamma_{k_1}(x) \cdots \gamma_{k_P}(x)\overline{\gamma_{l_1}} \cdots \overline{\gamma_{l_N}}$, and $\phi = \phi_{k_1} + \cdots + \phi_{k_P} - (\phi_{l_1} + \cdots + \phi_{l_N})$,
\[
	\left|\int_{a}^{b} \Big((\phi - K\eta)e^{Ki(\eta t + 2\theta(t,\eta))}e^{-i\phi t}\Gamma(t) - 2Ke^{Ki(\eta t + 2 \theta(t,\eta))}e^{-i\phi t}\Gamma(t)\frac{\partial\theta}{\partial x}(t,\eta) \Big)dt \right| \leq 2\tau^I,
\]
where $\tau = \sup_k \Var(\gamma_k,(0,\infty))$.
\end{lemma}
By \eqref{uniformly.b.v.condition}, $\tau < \infty$, and thus by \eqref{theta.prime}, each application of Lemma \ref{exchange.lemma} appends a new $L^p$ factor at a finite cost. 
Applying Lemma \ref{exchange.lemma} repeatedly to \eqref{log.r} yields terms of the form 
\begin{align}
\label{post.exchange.form}
	f(\eta; [\phi_{k_j}]_{j=1}^P; [\phi_{l_j}]_{j=1}^N) \int_{0}^{x} e^{Ki(\eta t + 2\theta(t,\eta))}\prod_{i=1}^{P} e^{-i\phi_{k_i} t} \gamma_{k_i}(t) \prod_{j=1}^{N} e^{i\phi_{l_j} t} \overline{\gamma_{l_j}(t)} dt, 
\end{align}
where $K = P-N \geq 0$. 
Once $I=P+N$ grows to $p$, the $L^p$ condition \eqref{uniformly.L.p.condition} gives a finite $x-$independent upper bound on that term. 

In \eqref{post.exchange.form} we see an important difference between the Schr\"odinger and Dirac settings. 
To ensure self-adjointness, the Schr\"odinger operator's potential is taken to be real-valued. 
Consequently, applications of Lemma \ref{exchange.lemma} yield the same terms as the above as well as terms with the complex conjugate of the product appearing in the integral. 
In the Dirac setting, since $\varphi$ is in general complex-valued, we inherit no such symmetry. 
This motivates the more cumbersome notation for the function $f(\eta)$ above. 
Going forward, the reader should think of the $P$ $\phi_{k_j}$ as the \textit{positive} frequencies and the $N$ $\phi_{l_j}$ as the \textit{negative} frequencies. 
While $P$ and $N$ depend on $I$ and $K$, we will often suppress this dependence for conciseness. 
Note also that the definitions of $I,K,P,$ and $N$ imply $P = \frac{I+K}{2}$ and $N = \frac{I-K}{2}$.

We also see a difference between the setting of orthogonal polynomials on the unit circle and the Dirac setting. 
In the Dirac setting, like in the Schr\"odinger setting, the appearance of $\partial_x \theta(x,\eta)$ leads to an increase or decrease of $K$, the number of $e^{i(\eta x + 2\theta(x,\eta))}$ factors, by one (or also zero, in the Schr\"odinger setting).
In the setting of orthogonal polynomials on the unit circle, on the other hand, the appearance of $\theta(n+1,\eta) - \theta(n,\eta)$ leads instead to a much wider range of $K$ values after the exchange. 
In \cite{L2011}, this is dealt with by passing to Taylor expansions of $e^{2ki(\theta_{n+1} - \theta_n)}$, but here we will be able to work directly with $\partial_x \theta(x,\eta)$. 

We will track the terms \eqref{post.exchange.form} as $I$ increases to $p$. 
Note that such terms appear for any permutation of $[\phi_{k_1}, \ldots, \phi_{k_P}]$ and for any permutation of $[\phi_{l_1}, \ldots, \phi_{l_N}]$, so we can agree to average $f$ over all such terms, by which we mean replacing $f(\eta; [\phi_{k_j}]_{j=1}^P; [\phi_{l_j}]_{j=1}^N)$ by 
\[
	\frac{1}{P!N!} \sum_{\substack{\sigma \in S_P \\ \tau \in S_N}} f(\eta; [\phi_{k_{\sigma(j)}}]_{j=1}^P; [\phi_{l_{\tau(j)}}]_{j=1}^N),
\] 
where $S_j$ denotes the symmetric group on $j$ elements. 
This averaging is useful both for avoiding counting the distinct permutations of the frequencies $\phi_j$ and, importantly, for showing that many apparent singularities arising in $f$ are removable, as we shall see in Section \ref{SectionNonremovableSingularities}. 
The transformed $f(\eta)$ is symmetric in the parameters $[\phi_{k_1}, \ldots, \phi_{k_P}]$ and in the parameters $[\phi_{l_1}, \ldots, \phi_{l_N}]$. 
For any given $I\geq 1$ and $0 \leq K \leq I$ and permutations $\sigma \in S_P$ and $\tau \in S_N$, there are either one or two types of term on which $f_{I,K}$ depends--$f_{I-1,K-1}$ and $f_{I-1,K+1}$--due to Lemma \ref{exchange.lemma} and the introduction via $\frac{\partial\theta}{\partial x}$ of both $e^{\pm i(\eta t + 2\theta)}$. 
The reason $f_{I,K}$ may only depend on one of these is that the term \eqref{post.exchange.form} does not appear if either $K<0$ or $K>I$, and we will define such $f_{I,K}$ to be zero.
The exchange allowed by Lemma \ref{exchange.lemma} motivates the following recursion: if $0 \leq K \leq I$ and as before $P = \frac{I+K}{2}, N = \frac{I-K}{2} \in \Z$, 
\begin{align}
	&f_{1,1}(\eta; [\phi_1]) = 1; \notag \\
	&g_{I,K}(\eta; [\phi_{k_j}]_{j=1}^P; [\phi_{l_j}]_{j=1}^N) = \frac{iK}{\sum_{j=1}^{P}\phi_{k_j} - \sum_{j=1}^{N}\phi_{l_j} - K\eta} f_{I,K}(\eta; [\phi_{k_j}]_{j=1}^P; [\phi_{l_j}]_{j=1}^N); \label{g.I.K} \\
	&f_{I,K}(\eta; [\phi_{k_j}]_{j=1}^P; [\phi_{l_j}]_{j=1}^N) = \frac{1}{P!N!} \sum_{a=-1}^{1} \sum_{\substack{\sigma \in S_P \\ \tau \in S_N}} \omega_a g_{I-1, K + a}(\eta; [\phi_{k_{\sigma(j)}}]_{j=1}^{\min[P, P+ a]}; [\phi_{l_{\tau(j)}}]_{j=1}^{\min[N, N-a]}), \label{f.I.K}
\end{align}
where we shall think of $\omega_a$ as a function of $1 + 1 + 0$ variables if $a=-1$ or of $1 + 0 + 1$ variables if $a=1$, in either case defined as
\[
	\omega_a = \omega_a(\eta; [\phi_{k_j}]_{j=1}^{\max [0,-a]}; [\phi_{l_j}]_{j=1}^{\max [0,a]}) = \delta_{a+1}-\delta_{a-1},
\]
where $\delta_j$ is one if $j=0$ and zero otherwise. 
By convention, we define $f_{I,K}$ and $g_{I,K}$ to be zero whenever $K >I$, $I-K \notin 2\Z$, or either $I<1$ or $K<0$, regardless of the number of frequencies on which they are made to depend. 			

Notation can be simplified using the following symmetric product: given $\mathfrak{f},$ a function of \mbox{$1 + P_1 + N_1$} variables, and $\mathfrak{g},$ a function of \mbox{$1 + P_2 + N_2$} variables, their symmetric product, $\mathfrak{f} \odot \mathfrak{g}$, is defined as
\begin{align*}
&(\mathfrak{f} \odot \mathfrak{g})(\eta; [\phi_{k_j}]_{j=1}^{P_1+P_2}; [\phi_{l_j}]_{j=1}^{N_1+N_2}) \\
&\hspace{.4in}= \frac{1}{(P_1+P_2)!(N_1+N_2)!} \sum_{\substack{\sigma \in S_{P_1+P_2} \\ \tau \in S_{N_1+N_2}}} \mathfrak{f}(\eta; [\phi_{k_{\sigma(j)}}]_{j=1}^{P_1}; [\phi_{l_{\tau(j)}}]_{j=1}^{N_1})\mathfrak{g}(\eta; [\phi_{k_{\sigma(j)}}]_{j=P_1+1}^{P_1+P_2}; [\phi_{l_{\tau(j)}}]_{j=N_1+1}^{N_1+N_2}).
\end{align*}
It is straightforward to check that $\odot$ is commutative and associative. 
We also define for $0 \leq K \leq I$
\[
	\Xi_{I,K}(\eta; [\phi_{k_j}]_{j=1}^P; [\phi_{l_j}]_{j=1}^N) := \delta_{I-1}\delta_{K-1}.
\]
With this notation, we can abbreviate the definition \eqref{f.I.K} as
\begin{align}
	f_{I,K}(\eta; [\phi_{k_j}]_{j=1}^P; [\phi_{l_j}]_{j=1}^N) &= \Xi_{I,K} + \sum_{a=-1}^{1} \omega_a \odot g_{I-1,K + a}. \label{pretty.f} 
\end{align} 
In each application of Lemma \ref{exchange.lemma}, the introduction of $\frac{\partial \theta}{\partial x} = \frac{i}{2}( e^{i(\eta x + \eta)} \varphi - \overline{e^{i(\eta x + \eta)} \varphi})$ yields two kinds of terms of the form \eqref{post.exchange.form}, one kind for each summand in $\frac{\partial\theta}{\partial x}$. 
In each kind, the value $K$, which is the number of $e^{i(\eta t + 2\theta)}$ factors in the integrand, has either increased or decreased by one, depending on which summand was attached. 
With enough decreases in $K$, the new $K$ may shrink to zero. 
Lemma \ref{exchange.lemma} may still be applied so long as the sum $\sum_{j=1}^{P} \phi_{k_j} - \phi_{l_j}$ is nonzero. 
When both $K=0$ and $\sum_{j=1}^{P} \phi_{k_j} - \phi_{l_j}=0$, however, Lemma \ref{exchange.lemma} cannot be applied, and we depend instead on the following cancellation:

\begin{lemma}
\label{zero.K.cancellation.lemma}
	For any $p \in \Z$ and $I \in 2\Z$, $f_{I,2p}$ is purely imaginary. 
	Moreover, if $I = 2P$ and \mbox{$\sum_{j=1}^{P} \phi_{k_j} - \sum_{j=1}^{P} \phi_{l_j} = 0$}, then
	\[
		f_{I,0}(\eta; [\phi_{k_j}]_{j=1}^P; [\phi_{l_j}]_{j=1}^P) = f_{I,0}(\eta; [\phi_{l_j}]_{j=1}^P; [\phi_{k_j}]_{j=1}^P).
	\]	  
\end{lemma}
\begin{proof}
	First we prove that $\Re f_{I,2p} = 0$ for all $p \in \Z$ by induction on $I = 2P$. 
	The case $P\leq 0$ holds trivially. 
	Suppose $\Re f_{2k,2p} = 0$ for all $p \in\Z$ and for all $k < P$. 
	By expanding $f_{I,2n}$ using \eqref{pretty.f} and \eqref{g.I.K}, we obtain
	\begin{align*}
		f_{I,2n} &= \sum_{a=-1}^{1} \omega_a \odot \frac{(2n+a)i}{\sum_{j=1}^{q_1} \phi_{k_j} - \sum_{j=1}^{r_1}\phi_{l_j} - (2n+a)\eta} \\
		&\hspace{1in}\times\big(\Xi_{I-1,2n+a} + \sum_{b=-1}^{1} \omega_b \odot \frac{(2n+a+b)i}{\sum_{j=1}^{q_2}\phi_{k_j} - \sum_{j=1}^{r_2}\phi_{l_j} - (2n+a+b)\eta} f_{I-2,2n+a+b}\big).
	\end{align*}
	In taking the real part, the term with the $\Xi$ factor vanishes due to the $i$ in the first quotient, and the other term vanishes due to the $i^2$ together with the induction hypothesis. 

	For the second part, we first prove by induction on $n$ that for $I = 2n$, $n \geq 1$,
	\begin{equation}
		f_{I,0} = \frac{1}{n!n!}\sum_{\sigma, \tau \in S_n} \sum_{s \in \mathcal{A}(I)} H_{I,s,\sigma, \tau}, \label{f.I.zero.via.H}
	\end{equation}
	where $\mathcal{A}(I)$ is the set of $I+1$-tuples $(s_0,s_1, \ldots, s_I)$ with integer components such that $|s_{i+1}-s_i|=1$, $s_i \geq 1$ for $1 \leq i \leq I-1$, and $s_0 = s_I = 0$, and 
	\begin{equation*}
		H_{I,s,\sigma, \tau}(\eta; [\phi_{k_j}]_{j=1}^{I/2}; [\phi_{l_j}]_{j=1}^{I/2}) = \prod_{m=1}^{I-1} \frac{i(s_m-s_{m-1})s_m}{\sum_{j=1}^{\frac{m+s_m}{2}} \phi_{k_{\sigma(j)}} - \sum_{j=1}^{\frac{m-s_m}{2}}\phi_{l_{\tau(j)}} - s_m \eta}. 
	\end{equation*}
	We obtained $f_{I,0}$ by averaging over permutations before each application of Lemma \ref{exchange.lemma}. 
	The terms $H_{I,s, \sigma, \tau}$ are obtained by applying Lemma \ref{exchange.lemma} $I-1$ times without averaging first. 
	We then return to $f_{I,0}$ by averaging over permutations. 
	We also have that
	\begin{align*}
		H_{I,s,\sigma, \tau}(\eta; [\phi_{k_j}]_{j=1}^{I/2}; [\phi_{l_j}]_{j=1}^{I/2}) = H_{I,\tilde{s}, \tau, \sigma}(\eta; [\phi_{l_j}]_{j=1}^{I/2}; [\phi_{k_j}]_{j=1}^{I/2}),
	\end{align*}
	where $\tilde{s}_i = s_{I-i}$. 
	Clearly, $s \in \mathcal{A}(I)$ if and only if $\tilde{s} \in \mathcal{A}(I)$, so summing over $\mathcal{A}(I)$ and averaging in permutations $\sigma, \tau \in S_n$ completes the proof. 
\end{proof}

With Lemma \ref{zero.K.cancellation.lemma}, we see that we may apply Lemma \ref{exchange.lemma} `to completion,' that is, until all terms remaining either have $p$ factors $\beta_j$, are bounded by $2\tau^K$ for some finite $K$, or else are purely imaginary and do not contribute to the Pr\"ufer amplitude. 
Though $\varphi$ not of finite type will yield infinitely many such terms, later we will see that \eqref{alpha.type.decay.condition} is sufficient to control the relevant infinite sums. 
Now, applying Lemma \ref{exchange.lemma} to completion introduces many of the $g_{I,K}$, each of which appears to introduce a singularity in $\eta$ at \mbox{$\frac 1K(\sum_{j=1}^{P}\phi_{k_j} - \sum_{j=1}^{N} \phi_{l_j})$}. 
In fact, such singularities for $K > 1$ are removable, due to the following lemma:

\begin{lemma}
\label{f.g.reduction.lemma}
If $0 < K \leq I$ and $0 < k < K$, then
\begin{align}
	f_{I,K} &= \sum_{i=0}^{I} f_{i,k} \odot g_{I-i,K-k}, \label{f.recursion}\\
	g_{I,K} &= \sum_{i=0}^{I} g_{i,k} \odot g_{I-i, K-k}. \label{g.recursion}
\end{align}
\end{lemma}
\begin{proof}
	The proof is similar to that of \cite[Lemma 5.1 (i)]{L2013SlowWvN}. 
	We prove \eqref{f.recursion} and \eqref{g.recursion} simultaneously by induction on $I$. 
	Both statements hold vacuously when $I< 2$ or $K <2$, so we assume $2 \leq K \leq I$. 
	Suppose both statements hold for all $\tilde{I} < I$. 
	Then using \eqref{pretty.f} and associativity of $\odot$ yields
	\begin{align*}
		\sum_{i=0}^{I} f_{i,k} \odot g_{I-i,K-k} &= \sum_{i=0}^{I} (\Xi_{i,k} + \sum_{a=-1}^{1} \omega_a \odot g_{i-1,k+a}) \odot g_{I-i, K-k} \\
		&= \sum_{i=0}^{I} \Xi_{i,k} \odot g_{I-i,K-k} + \sum_{a=-1}^{1} \omega_a \odot (\sum_{i=0}^{I}g_{i-1,k+a} \odot g_{I-i,K-k}).
	\end{align*}
	Since by our convention $g_{-1,k+a}=0$ for any $k,a$, we may reindex in $i$ with no cost as
	\[
		\delta_{k-1}\Xi_{1,1} \odot g_{I-1,K-1} + \sum_{a=-1}^{1} \omega_a \odot (\sum_{i=0}^{I-1} g_{i,k+a} \odot g_{I-1-i,K-k}).
	\]	
	At this point, we may apply the induction hypothesis on the inner sum both when $a=-1$ and $a=1$ so long as $0 < k+a< K+a$ and $0 < K+a \leq I$. 
	If $K+a > I$, then for each $i$, either $g_{i,k+1}$ or $g_{I-1-i,K-k}$ is zero. 
	Since $g_{I-1,K+a}$ is also zero, we may include this term at no cost. 
	That $0 < k< K$ implies $k+1<K+a$, but we do have $k+a \leq 0$ exactly when $k=1$ and $a=-1$. 
	This is the only exception we must make to the induction hypothesis, and we're left with
	\begin{align*}
		\sum_{i=0}^{I} f_{i,k} \odot g_{I-i,K-k} &= \delta_{k-1}\Xi_{1,1} \odot g_{I-1,K-1} + \sum_{a=-1}^{1} \omega_a \odot (g_{I-1,K+a} - \delta_{k-1}\delta_{a+1}g_{I-1,K-1}) \\
		&= f_{I,K} - \Xi_{I,K} + \delta_{k-1} (\Xi_{1,1} \odot g_{I-1,K-1} - \omega_{-1} \odot g_{I-1,K-1}) = f_{I,K},
	\end{align*}
	where we've used $\Xi_{I,K}=0$ (since $I >1$) and $\omega_{-1} = \Xi_{1,1}$. 

	It remains to show \eqref{g.recursion} at $I$. 
	As before, let $I = P + N$ and $K = P - N$. 
	Now also set $i = P_{i,k} + N_{i,k}$ and $k = P_{i,k} - N_{i,k}$ for any $1 \leq i \leq I$ and $0 < k < K$. 
	By \eqref{g.I.K}, for any permutations $\sigma \in S_P$ and $\tau \in S_N$ and for any $1 \leq i \leq I$ and $0 < k < K$, we have
	\begin{align*}
		\frac{K f_{I,K}(\eta; [\phi_{k_{\sigma(j)}}]_{j=1}^P; [\phi_{l_{\tau(j)}}]_{j=1}^N)}{g_{I,K}(\eta; [\phi_{k_{\sigma(j)}}]_{j=1}^P; [\phi_{l_{\tau(j)}}]_{j=1}^N)} &= \frac{k f_{i,k}(\eta; [\phi_{k_{\sigma(j)}}]_{j=1}^{P_{i,k}}; [\phi_{l_{\tau(j)}}]_{j=1}^{N_{i,k}})}{g_{i,k}(\eta; [\phi_{k_{\sigma(j)}}]_{j=1}^{P_{i,k}}; [\phi_{l_{\tau(j)}}]_{j=1}^{N_{i,k}})} \\
		&\hspace{1in}+ \frac{(K-k) f_{I-i,K-k}(\eta; [\phi_{k_{\sigma(j)}}]_{j=P_{i,k}+1}^P; [\phi_{l_{\tau(j)}}]_{j=N_{i,k}+1}^N)}{g_{I-i,K-k}(\eta; [\phi_{k_{\sigma(j)}}]_{j=P_{i,k}+1}^P; [\phi_{l_{\tau(j)}}]_{j=N_{i,k}+1}^N)}.
	\end{align*}
	Clearing denominators here and averaging in permutations $(\sigma,\tau) \in S_P \times S_N$, notice that each of $f_{I,K}$ and $g_{I,K}$ is symmetric with respect to $(\sigma,\tau)$, so that we're left with 
	\begin{align*}
		\frac{Kf_{I,K}}{g_{I,K}} g_{i,k} \odot g_{I-i,K-k} &= \frac{1}{P!N!} \sum_{\substack{\sigma \in S_P \\ \tau \in S_N}} \big(kf_{i,k}(\eta; [\phi_{k_{\sigma(j)}}]_{j=1}^{P_{i,k}}; [\phi_{l_{\tau(j)}}]_{j=1}^{N_{i,k}}) g_{I-i,K-k}(\eta; [\phi_{k_{\sigma(j)}}]_{j=P_{i,k}+1}^P; [\phi_{l_{\tau(j)}}]_{j=N_{i,k}+1}^N) \\
		&\hspace{.5in}+ (K-k)f_{I-i,K-k}(\eta; [\phi_{k_{\sigma(j)}}]_{j=P_{i,k}+1}^P; [\phi_{l_{\tau(j)}}]_{j=N_{i,k}+1}^N) g_{i,k}(\eta; [\phi_{k_{\sigma(j)}}]_{j=1}^{P_{i,k}}; [\phi_{l_{\tau(j)}}]_{j=1}^{N_{i,k}})\big) \\
		&=kf_{i,k} \odot g_{I-i,K-k}  + (K-k) f_{I-i,K-k} \odot g_{i,k}.
	\end{align*}
	Summing both sides in $0 \leq i \leq I$ we have
	\begin{align*}
		\frac{Kf_{I,K}}{g_{I,K}} \sum_{i=0}^{I} g_{i,k} \odot g_{I-i,K-k} &= k\sum_{i=0}^{I} f_{i,k} \odot g_{I-i,K-k} + (K-k)\sum_{i=0}^{I} f_{I-i,K-k} \odot g_{i,k} \\
		% &= kf_{I,K} + (K-k) \sum_{i=0}^{I} f_{i,K-k} \odot g_{I-i,k} \\
		&= kf_{I,K} + (K-k) \sum_{i=0}^{I} f_{i,l} \odot g_{I-i,K-l} = Kf_{I,K},
	\end{align*}
	where $l := K-k$ and we've used \eqref{f.recursion}, since $0 < l < K$ also. Dividing both sides by $\frac{Kf_{I,K}}{g_{I,K}}$ completes the proof.
\end{proof}

Lemma \ref{f.g.reduction.lemma} allows us to reduce $f_{I,K}$ and $g_{I,K}$ to sums of products of only the $f_{i,1}$ and $g_{i,1}$. 
We may use this reduction and a study of singularities of $g_{i,1}$ to define the correct exceptional set $S_p$. 
Note that $I-K\notin 2\Z$ implies $g_{I,K}=0$ by our convention and so contributes no singularities. 

\begin{lemma}
\label{nonremovable.singularities.g} \
	\begin{enumerate}
		\item
		For $I = 2n-1$, $n \geq 1$, if $\zeta$ is a nonremovable singularity of $g_{I,1}(\eta; [\phi_{k_j}]_{j=1}^n;[\phi_{l_j}]_{j=1}^{n-1})$, then $\zeta$ may be written in the form \eqref{form.of.singularities} for $1 \leq m \leq n$. 
		\item 
		For $1 < K \leq I$, $I - K \in 2\Z$, if $\zeta$ is a nonremovable singularity of $g_{I,K}$, then $\zeta$ may be written in the form \eqref{form.of.singularities} for $1 \leq m \leq n$, where $n = \lfloor I/2 \rfloor$.
	\end{enumerate}
\end{lemma}
\begin{proof} \
	\begin{enumerate}
		\item The case $n=1$ follows immediately from \eqref{g.I.K}. Suppose the statement holds for $n < N$ and let $I = 2N-1$. By \eqref{g.I.K}, $g_{I,1}$ has $\sum_{j=1}^{N}\phi_{k_j} - \sum_{j=1}^{N-1}\phi_{l_j}$ as a nonremovable singularity and all other nonremovable singularities are those arising from $f_{I,1}$. By \eqref{pretty.f}, singularities of $f_{I,1}$ arise from $g_{I-1,2}$, and by \eqref{g.recursion}, these are just the singularities arising from $g_{i,1}$ for $0 \leq i \leq I-1$. The induction hypothesis completes the proof. 
		\item Fix $2 \leq K \leq I$. By iteratively applying \eqref{g.recursion} with $k=1$ to $g_{I,K}$ $K-1$ times, we see that nonremovable singularities of $g_{I,K}$ arise as those of $g_{i,1}$ for $1 \leq i \leq I-1$, since for $i =0, I,$ either $g_{i,\cdot}$ or $g_{I-i,\cdot}$ is zero. By part (i), if $I = 2n$, the nonremovable singularities of $g_{i,1}$ for $1 \leq i \leq I-1$ are of the form \eqref{form.of.singularities} for $1 \leq m \leq n$. If $I = 2n+1$, then $g_{I-1,1}$ is zero, so nonremovable singularities of $g_{i,1}$ for $1 \leq i \leq I-1$ are still of the form \eqref{form.of.singularities} for $1 \leq m \leq n$. In either case, $n = \lfloor I/2 \rfloor.$ 
	\end{enumerate}
\end{proof}

\begin{corollary}
\label{nonremovable.singularities.corollary}
For $0 \leq K \leq I$, $I-K \in 2\Z$, if $\zeta$ is a nonremovable singularity of $f_{I,K}$, then $\zeta$ may be written in the form \eqref{form.of.singularities} for $ 1 \leq m \leq n$, where $n = \lfloor (I-1)/2 \rfloor$. 
\end{corollary}
\begin{proof}
	If $I=2n+1$ and $K=0$ or $K=1$, we have from \eqref{pretty.f}
	\[
		f_{I,K} = \Xi_{I,1} + \omega_1 \odot g_{I-1,K+1},
	\]	
	since $g_{I-1,K-1} = 0$. Thus, nonremovable singularities of $f_{I,K}$ are the same as those of $g_{I-1,K+1}$. By Lemma \ref{nonremovable.singularities.g} (ii), each of these has singularities of the form \eqref{form.of.singularities} for $1 \leq m \leq \lfloor (I-1)/2 \rfloor = n$. 

	On the other hand, if $1 < K \leq I$, $I-K \in 2\Z$, we can use \eqref{f.recursion} to write $f_{I,K}$ as 
	\[
		f_{I,K} = \sum_{i=0}^{I} f_{i,1} \odot g_{I-i,K-1},
	\]	 
	from which we see that the nonremovable singularities of $f_{I,K}$ arise from those of $f_{i,1}$ and $g_{i,K-1}$ for $1 \leq i \leq I$. Singularities of $f_{i,1}$ are of the form \eqref{form.of.singularities} for $1 \leq m \leq \lfloor (I-1)/2 \rfloor$ by the first part of the proof (since for $I$ even, $f_{I,1}=0$). Singularities of $g_{i,K-1}$ are of the form \eqref{form.of.singularities} for $1 \leq m \leq \lfloor (I-1)/2 \rfloor$ by Lemma \ref{nonremovable.singularities.g} (ii) (since $g_{I,K-1}=0$ by the assumption $I-K \in 2\Z$).
\end{proof}
 
Now that we have seen that the set $S_p$ is decided by only the functions $g_{I,1}$, we relate $g_{I,1}$ and the functions $h_I$ from \eqref{h.I} in anticipation of the proof of Lemma \ref{crux.lemma}. 

\begin{lemma}
\label{g.to.h}
The functions $g_{I,1}$ for $I = 2P-1$ are just rescaled, symmetrized $h_I$, namely,
\[
	g_{I,1}(\eta; [\phi_{k_j}]_{j=1}^P; [\phi_{l_j}]_{j=1}^{P-1}) = \frac{i}{P!(P-1)!} \sum_{\substack{\sigma \in S_P \\ \tau \in S_{P-1}}} h_I(\eta; [\phi_{k_{\sigma(j)}}]_{j=1}^P; [\phi_{l_{\tau(j)}}]_{j=1}^{P-1}).
\]
\end{lemma}
\begin{proof}
	We prove the lemma by induction on $P$. The case $P=1$ holds by definition. Suppose the statement holds up to $P-1$. Let us call $\phi = \sum_{j=1}^{P} \phi_{k_j} - \sum_{j=1}^{P-1} \phi_{l_j}$. Combining \eqref{g.I.K}, \eqref{f.I.K}, and \eqref{g.recursion}, we have
	\begin{align*}
		&g_{I,1}(\eta; [\phi_{k_j}]_{j=1}^P; [\phi_{l_j}]_{j=1}^{P-1})= \frac{-i}{\phi - \eta} \sum_{m=0}^{I-1} \frac{1}{P!(P-1)!} \sum_{\substack{\sigma \in S_{P} \\ \tau \in S_{P-1}}} g_{m,1}(\eta; [\phi_{k_{\sigma(j)}}]_{j=1}^{\frac{m+1}{2}}; [\phi_{l_{\tau(j)}}]_{j=1}^{\frac{m-1}{2}}) \\
		&\hspace{4in}\odot g_{I-1-m, 1}(\eta; [\phi_{k_{\sigma(j)}}]_{j=\frac{m+3}{2}}^P; [\phi_{l_{\tau(j)}}]_{j=\frac{m+1}{2}}^{P-1})
	\end{align*}
	Then we use the induction hypothesis to rewrite as
	\begin{align*}
		&\frac{-i}{\phi - \eta} \sum_{m=0}^{I-1} \frac{1}{P!(P-1)!} \sum_{\substack{\sigma \in S_{P} \\ \tau \in S_{P-1}}} \frac{i}{(\frac{m+1}{2})!(\frac{m-1}{2})!} \sum_{\substack{\mu \in S_{\frac{m+1}{2}} \\ \nu \in S_{\frac{m-1}{2}}}} h_m(\eta; [\phi_{k_{\mu(\sigma(j))}}]_{j=1}^{\frac{m+1}{2}}; [\phi_{l_{\nu(\tau(j))}}]_{j=1}^{\frac{m-1}{2}}) \\
		&\hspace{2in} \odot \frac{i}{(\frac{I-m}{2})!(\frac{I-m-2}{2})!} \sum_{\substack{\alpha \in S_{\frac{I-m}{2}} \\ \beta \in S_{\frac{I-m-2}{2}}}} h_{I-1-m}(\eta; [\phi_{k_{\alpha(\sigma(j))}}]_{j=\frac{m+3}{2}}^P; [\phi_{k_{\beta(\tau(j))}}]_{j=\frac{m+1}{2}}^{P-2}) \\
		&\hspace{.5in}= \frac{i}{P!(P-1)!} \sum_{\substack{\sigma \in S_P \\ \tau \in S_{P-1}}} \frac{1}{\sum_{j=1}^{P} \phi_{k_{\sigma(j)}}  - \sum_{j=1}^{P-1}\phi_{l_{\tau(j)}} - \eta} \\
		&\hspace{1.5in}\times \sum_{m=0}^{I-1} h_m(\eta; [\phi_{k_{\sigma(j)}}]_{j=1}^{\frac{m+1}{2}}; [\phi_{l_{\tau(j)}}]_{j=1}^{\frac{m-1}{2}}) h_{I-1-m}(\eta; [\phi_{k_{\sigma(j)}}]_{j=\frac{m+3}{2}}^P; [\phi_{l_{\tau(j)}}]_{j=\frac{m+1}{2}}^{P-2}), 
	\end{align*}
	where we've noticed that the permutations $\mu, \nu, \alpha,$ and $\beta$ as well as the remaining $\odot$ are redundant since the $h_m$ and $h_{I-1-m}$ have already been symmetrized. The definition of $h_I$ completes the proof. 
\end{proof}
\end{section}

\begin{section}{Finitely many summands} \label{SectionFinitelyManySummands}
If $\varphi$ is of finite Wigner-von Neumann type, there are finitely many nonremovable singularities arising from applications of Lemma \ref{exchange.lemma}, all of the form \eqref{form.of.singularities}. We obtain the first of our main results:

\begin{proof}[Proof of Theorem \ref{finite.main.thm}]
By Lemma \ref{stolz.type.lemma} and subordinacy theory for Dirac operators due to Behncke, it is enough to show that given $\frac{\eta}{2} \notin S_p$, all solutions $U(x,\eta)$ are bounded, regardless of their boundary value at zero. Given such a solution $U$, we pass to its Pr\"ufer amplitude $r(x,\eta)$ and begin from \eqref{log.r}. Repeatedly applying Lemma \ref{exchange.lemma} to 
\[
	\int_{0}^{x} e^{i(\eta t + 2\theta(t,\eta))} \varphi(t)dt,
\]
where $\varphi$ is of the form \eqref{WvNPotential}, we obtain a finite sum of terms of the form \eqref{post.exchange.form} with either $I=p$ or $K=0$ and a finite sum of errors, with each error bounded by $2\tau^K$ for some finite $K$. The leading functions $f_{I,K}(\eta; [\phi_{k_j}]_{j=1}^P; [\phi_{l_j}]_{j=1}^N)$ are meromorphic functions in $\eta$ where all poles have the form \eqref{form.of.singularities} with $1 \leq m \leq n$ by Lemma \ref{nonremovable.singularities.corollary}. The terms with $I=p$ are uniformly bounded in $x$ by the $L^p$ condition on the $\beta_j$. For one of the (finitely many) terms with $K=0$,
\[
	f_{2P,0}(\eta; [\phi_{k_j}]_{j=1}^P; [\phi_{l_j}]_{j=1}^P) \int_{0}^{x} \prod_{j=1}^{P} e^{-i\phi_{k_j} t} e^{i\phi_{l_j} t} \gamma_{k_j}(t)  \overline{\gamma_{l_j}(t)} dt,
\] 
if $\sum_{j=1}^{P}\phi_{k_j}-\phi_{l_j} \neq 0$, we may apply Lemma \ref{exchange.lemma} once more to give a finite $x$-independent upper bound. If $\sum_{j=1}^{P}\phi_{k_j}-\phi_{l_j} = 0$, then there is a corresponding term,
\[
	f_{2P,0}(\eta; [\phi_{l_j}]_{j=1}^P; [\phi_{k_j}]_{j=1}^P) \int_{0}^{x} \prod_{j=1}^{P} \overline{\gamma_{k_j}(t)}  \gamma_{l_j}(t) dt,
\]
where the two constants $f_{I,0}(\eta)$ of these two corresponding terms are equal and purely imaginary by Lemma \ref{zero.K.cancellation.lemma}. The sum of these corresponding terms, then, is purely imaginary, and taking the real part in \eqref{log.r} annihilates them both. This shows that $\log r(x,\eta)$ is bounded uniformly in $x \in [0,\infty)$ and completes the proof. 
\end{proof}

We now turn to the proof of Theorem \ref{example.thm} and show that elements of $S_p$ may indeed appear as embedded eigenvalues. We begin with a solution $U(x,\eta)$ to \eqref{eigenequation} for $\varphi$ of the form \eqref{special.operator.data.form}.
Consider $r(x,\eta)$ at ${\eta_0 = \sum_{j=1}^{n} \phi_{k_j} - \sum_{j=1}^{n-1} \phi_{l_j}}$, where $p = 2n+1$ and assume that $\eta_0$ cannot be rewritten in the form \eqref{form.of.singularities} using fewer frequencies. We will use the following analog of \cite[Lemma 6.1]{L2013SlowWvN}:

\begin{lemma}
\label{special.lemma}
Let $E \in \R$ and let $r(x,\eta), \theta(x,\eta)$ be the Pr\"ufer variables for a solution $U(x,\eta)$ of \eqref{eigenequation} at $E = \eta/2$. If 
\begin{equation}
	\partial_x \log r(x,\eta) = -\frac{B}{x^{(p-2)\gamma}} + b(x,\eta)  \label{power.decaying.log.r}
\end{equation}
for some $b(x,\eta)$ integrable in $x$ and $\theta_\infty = \xlim \theta(x)$ exists, then for some $A>0$, we have
\begin{equation}
	U(x) = Af(x) \binom{(1+i) e^{-i(\eta/2 x + \theta_\infty)}}{(1-i) e^{i(\eta /2 x + \theta_\infty)}}(1 + o(1)), \ x \to \infty, \label{asymptotics.one}
\end{equation}
where
\begin{equation}
	f(x) = 
	\begin{cases}
		x^{-B} & \text{ if } \gamma = \frac{1}{p-2} \\
		\exp(-\frac{B}{1 - (p-2)\gamma} x^{1-(p-2)\gamma}) & \text{ if } \gamma \in (\frac{1}{p}, \frac{1}{p-2}).
	\end{cases} \label{asymptotics.two}
\end{equation}	
\end{lemma}
\begin{proof}
	Follows from the proof of \cite[Lemma 6.1]{L2013SlowWvN} and \eqref{prufer.variables.defn}. 
\end{proof}

Beginning from \eqref{log.r}, we may apply our iterative procedure from the proof of Theorem \ref{finite.main.thm} to produce terms with $x$-independent upper bounds and other terms of the form 
\[
	f_{2n-1,1}(\eta_0; [\phi_{k_j}]_{j=1}^n; [\phi_{l_j}]_{j=1}^{n-1}) x^{-\delta(p-2)} e^{i(\sum_{j=1}^{n}\xi_{k_j}(x) - \sum_{j=1}^{n-1}\xi_{l_j}(x) + 2\theta)} \prod_{j=1}^{n}c_{k_j} \prod_{j=1}^{n-1}\overline{c_{l_j}}.
\] 
We cannot apply Lemma \ref{exchange.lemma} to these latter terms due to the nonremovable singularity in $g_{2n-1,1}(\eta)$ at $\eta_0$. 
These terms appear once with each distinct permutation pair $(\sigma, \tau)$ of $(k_1, \ldots, k_n) \times (l_1, \ldots, l_{n-1})$, and since there may be repeated indices $k_j = k_i$ or $l_j = l_i$, we simply denote the number of such distinct permutations by $C$. 
Thus we arrive at 
\begin{equation}
	\partial_x \log r(x, \eta) = \Re \big(\frac{\Lambda}{ x^{\delta(p-2)}} e^{i(\xi(x) + 2\theta)} + b(x)\big), \label{special.log.r}
\end{equation}
where $\xi(x) = \sum_{j=1}^{n}\xi_{k_j}(x) - \sum_{j=1}^{n-1}\xi_{l_j}(x)$, $b \in L^1( (0,\infty))$, and
\[
	\Lambda = Cf_{2n-1,1} \prod_{j=1}^{n}c_{k_j} \prod_{j=1}^{n-1}\overline{c_{l_j}}.
\]
Without the exponential term in the above, the right hand side is clearly not integrable. 
Our goal is to choose $\xi$ such that the exponential term above does not oscillate enough to make the right hand side integrable. 
To that end, we wish to control the behavior of $\frac{\partial\theta}{\partial x}$. 
Our iterative procedure starting with \eqref{theta.prime} leaves the same terms from before with the addition of terms of the form 
\[
	f_{2P,0}(\eta; [\phi_{k_j}]_{j=1}^P; [\phi_{l_j}]_{j=1}^P) \int_{0}^{x} \prod_{j=1}^{P} \gamma_{k_j}(t)  \overline{\gamma_{l_j}(t)} dt,
\]
where $\sum_{j=1}^{P} \phi_{k_j} - \phi_{l_j} = 0$, for all $1 \leq P \leq n$. 
These terms were eliminated in the proof of Theorem \ref{finite.main.thm} by taking the real part since we were after $\partial_x\log r(x,\eta)$, but in controlling $\frac{\partial\theta}{\partial x}$, we take the imaginary part, and these terms must be included. 
Thus we write 
\[
	\partial_x \theta(x,\eta) = -\Im \big(\Omega(x) + \frac{\Lambda}{x^{\delta(p-2)}} e^{i(\xi(x) + 2\theta)} + c(x)\big),
\]
where again $c(x) \in L^1$ and now
\[
	\Omega(x) = \sum_{P=1}^{n} \sum_{\sum_{j=1}^{P}\phi_{k_j} - \phi_{l_j} = 0} f_{2P,0}(\eta; [\phi_{k_j}]_{j=1}^P; [\phi_{l_j}]_{j=1}^P) \prod_{j=1}^{P} \gamma_{k_j}(x)  \overline{\gamma_{l_j}(x)}.
\]
For the moment, let us consider only the case $p=3$ for convenience. 
Then $\Omega$ simplifies, after appealing to \eqref{f.I.K} to compute $f_{2,0}$, to
\begin{equation}
	\Omega(x) = -i\sum_{j=1}^{M}  \frac{|c_j|^2}{\phi_j - \eta} x^{-2\delta}. \label{p.equals.3.omega}
\end{equation}
Supposing $\min\{\phi_j: 1\leq j\leq M\} < \eta < \max\{\phi_j: 1\leq j \leq M\}$, we may choose the $c_j$ such that the summands in $\Omega$ cancel and $\Omega(x) \equiv 0$. 
Then, using the proof of \cite[Lemma 6.3]{L2013SlowWvN} to instead arrive at a choice of $\xi(x)$ yielding
\[
	\xlim \xi(x) + 2\theta(x) = -\arg \Lambda,
\]
rather than $-\frac{\pi}{2} - \arg \Lambda$ (which requires no change to the proof), appealing to \eqref{special.log.r} and Lemma \ref{special.lemma} shows $\eta_0$ is an embedded eigenvalue in the case $p=3$. 
This shows that the sets $S_p$ indeed contain possible embedded eigenvalues, but does not show that the growth in the sets $S_p$ is necessary. 
For $p = 5$, $\eta_0 \in S_5\setminus S_3$, and to prove Theorem \ref{example.thm} we show that $\eta_0$ in this case can be an embedded eigenvalue.

\begin{proof}[Proof of Theorem \ref{example.thm}]
When $p=5$, $\Omega$ includes the terms in \eqref{p.equals.3.omega} along with the terms
\[
	\sum_{\sum_{j=1}^{2}\phi_{k_j} - \phi_{l_j} = 0} f_{4,0}(\eta; [\phi_{k_j}]_{j=1}^2; [\phi_{l_j}]_{j=1}^2) \prod_{j=1}^{2} \gamma_{k_j}(x)  \overline{\gamma_{l_j}(x)} = \sum_{j_1, j_2 = 1}^M f_{4,0}(\eta; [\phi_{j_1}, \phi_{j_2}]; [\phi_{j_1}, \phi_{j_2}]) |c_{j_1}c_{j_2}|^2 x^{-4\delta}.
\]
Then from \eqref{f.I.K} we compute
\[
	f_{4,0}(\eta; [\phi_{j_1}, \phi_{j_2}]; [\phi_{j_1}, \phi_{j_2}]) = -\frac{i}{2} \frac{\phi_{j_1} + \phi_{j_2} - 2\eta}{(\phi_{j_1}-\eta)^2(\phi_{j_2}-\eta)^2}.
\]
Thus, if conditions \eqref{second.order.condition} and \eqref{fourth.order.condition} both hold, $\Omega(x)$ is identically zero, and the rest of the proof follows exactly as in the $p=3$ case.
\end{proof}

Lastly, we show that conditions \eqref{second.order.condition} and \eqref{fourth.order.condition} can hold simultaneously, so that Theorem \ref{example.thm} can hold non-vacuously. 
We take $M=3$ in \eqref{special.operator.data.form}, so that 
\[
	\varphi(x) = ae^{i\phi x} \gamma_1(x) + be^{i\psi x} \gamma_2(x) + ce^{i\rho x} \gamma_3(x),
\]
where $\gamma_j(x) = e^{i\xi_j(x)}x^{-\delta}$ and $a,b,c \in \C$. The conditions \eqref{second.order.condition} and \eqref{fourth.order.condition} become 
\begin{align}
	&\frac{|a|^2}{\phi - \eta} + \frac{|b|^2}{\psi - \eta} + \frac{|c|^2}{\rho - \eta} = 0, \label{explicit.second.order.condition}\\
	&\frac{|a|^4}{(\phi-\eta)^3} + \frac{|b|^4}{(\psi-\eta)^3} + \frac{|c|^4}{(\rho-\eta)^3} + \frac{|ab|^2}{(\phi-\eta)^2(\psi-\eta)^2}(\phi+\psi - 2\eta)  \notag\\
	&\hspace{1in}+ \frac{|ac|^2}{(\phi-\eta)^2(\rho - \eta)^2}(\phi+\rho - 2\eta) +\frac{|bc|^2}{(\psi-\eta)^2(\tau-\eta)^2}(\psi+\rho - 2\eta)= 0. \label{explicit.fourth.order.condition}
\end{align}
Let $E = \eta/2$ for $\eta = \phi+ \psi - \rho$. 
Then $\phi,\psi, \rho$ rationally independent implies $E \in S_5 \setminus S_3$. 
We wish to choose $\phi, \psi,$ and $\tau$ rationally independent such that the above conditions hold. Condition \eqref{explicit.second.order.condition} is equivalent to $|c|^2 = \frac{|a|^2}{\psi-\rho} + \frac{|b|^2}{\phi-\rho},$ and we assume $\psi < \rho < \phi$ so that this condition may be satisfied for many choices of $a$ and $b$. Suppose $2\rho - \phi - \psi = 1$. This implies the following simple identities:
\begin{center}
\begin{tabular}{ll}
$\phi - \eta = \rho - \psi,$ & $\phi + \psi - 2\eta = 1,$ \\
$\psi - \eta = \rho - \phi,$ & $\phi + \rho - 2\eta = \rho - \psi + 1,$ \\
$\rho - \eta = 1,$ & $\psi + \rho - 2 \eta = \rho - \phi + 1,$ 
\end{tabular}
\end{center}
which allow us to rewrite \eqref{explicit.fourth.order.condition} as
\begin{align*}
	&\frac{|a|^4}{(\rho - \psi)^3} + \frac{|b|^4}{(\rho - \phi)^3} + (\frac{|a|^2}{\psi-\rho} + \frac{|b|^2}{\phi-\rho})^2 + \frac{|ab|^2}{(\rho - \psi)^2(\rho - \phi)^2} \\
	&\hspace{1in}+ \frac{|a|^2}{(\rho - \psi)^2}(\rho - \psi + 1)(\frac{|a|^2}{\psi-\rho} + \frac{|b|^2}{\phi-\rho}) +\frac{|b|^2}{(\rho - \phi)^2}(\rho - \phi + 1)(\frac{|a|^2}{\psi-\rho} + \frac{|b|^2}{\phi-\rho})= 0,
\end{align*}
which equality holds identically, independently of the choice of $a$ and $b$, just by expanding and cancelling. 
There are many choices of $\phi, \psi,$ and $\rho$ so that $\psi < \rho < \phi$, $2\rho - \phi - \psi = 1$, and they are rationally independent. 
We may take, for example, $\phi = \sqrt{5}, \rho = \sqrt{3}$, and $\psi = 2\sqrt{3} - \sqrt{5} -1$. 
For any such choice, both conditions \eqref{explicit.second.order.condition} and \eqref{explicit.fourth.order.condition} hold, and Theorem \ref{example.thm} implies that $\phi + \psi - \rho$ is an embedded eigenvalue. 

\end{section}

\begin{section}{Infinitely many summands} \label{SectionInfinitelyManySummands}
If $\varphi$ is not of finite type, our goal is instead to bound the Hausdorff measure of $S_p$. To start this section, we prove Lemma \ref{crux.lemma} in a series of smaller lemmas. Our strategy is to begin with 
\[
	\log \frac{Z(x,\eta)}{Z(0,\eta)} = \int_{0}^{x} \mathcal{S}_{1,1}(t) dt,
\]
where 
\begin{equation}
	\mathcal{S}_{I,K}(x) = \sum_{\substack{k_1, \ldots, k_P = 1 \\ l_1, \ldots, l_N =1}}^\infty f_{I,K}(\eta; [\phi_{k_j}]_{j=1}^P; [\phi_{l_j}]_{j=1}^N) \prod_{j=1}^{P} \beta_{k_j}(x) \prod_{j=1}^{N} \overline{\beta_{l_j}(x)} e^{iK(\eta x + 2\theta(x))},
\end{equation}
with $\beta_j(x) = c_j e^{-i\phi_j x} \gamma_j(x)$, and then pass to higher values of $I$ via Lemma \ref{exchange.lemma}. Note that $\mathcal{S}_{I,K}$ is trivial if $I + K \notin 2\Z$ or $I<K$. We will track errors using
\[
	E_{I,K} = 2\sum_{\substack{k_1, \ldots, k_{P_K} = 1 \\ l_1, \ldots, l_{N_K} = 1}}^\infty |g_{I,K}\prod_{j=1}^{P}c_{k_j}\prod_{j=1}^{N}c_{l_j}|.
\]

\begin{lemma}
\label{E.finite.lemma}
	If $\varphi$ obeys the conditions of Lemma \ref{crux.lemma}, then $E_{I,K}$ is finite for $1 \leq K \leq I \leq p-2$.
\end{lemma}
\begin{proof}
	By Lemma \ref{g.to.h} and \eqref{h.I}, since condition \eqref{small.divisors.condition} of Lemma \ref{crux.lemma} holds for $1\leq I \leq p-2$, $E_{I,1}$ is finite for the same values of $I$. Then \eqref{g.recursion} gives
	\[
		E_{I,K} \leq \sum_{i=0}^{I} E_{i,k}E_{I-i,K-k},
	\]	
	for any $0 < k < K$. The lemma follows. 
\end{proof}

\begin{lemma}
\label{abs.conv.lemma}
	If $\varphi$ obeys the conditions of Lemma \ref{crux.lemma}, then the sum $\mathcal{S}_{I,K}(t)$ is absolutely convergent when $0 \leq K \leq I \leq p$. If, moreover, $I=p$, then
	\[
		\int_{0}^{\infty} \sum_{\substack{k_1, \ldots, k_P =1 \\ l_1, \ldots, l_N = 1}}^\infty |f_{I,K}(\eta; [\phi_{k_j}]_{j=1}^P;[\phi_{l_j}]_{j=1}^N) \prod_{j=1}^{P}\beta_{k_j} \prod_{j=1}^{N}\overline{\beta_{l_j}}| dt \leq \sum_{a=-1}^{1} |\omega_a|E_{p-1,K+a} \sum_{j=1}^{\infty}|c_j|\sigma^p.
	\]
\end{lemma}
\begin{proof}
	Taking absolute values in \eqref{pretty.f} gives
	\[
		|f_{I,K}| \leq |\Xi_{I,K}| + \sum_{a=-1}^{1}|\omega_a| \odot |g_{I-1,K+a}|.
	\]
	Multiplying by 
	\[
		|\prod_{j=1}^{P}\beta_{k_j}\prod_{j=1}^{N}\overline{\beta_{l_j}}| \leq |\prod_{j=1}^P c_{k_j} \prod_{j=1}^{N}c_{l_j}|\tau^I
	\]
	and summing in $k_j$ and $l_j$ proves absolute convergence. If $I=p$, we instead multiply by 
	\[
		\int_{0}^{\infty} |\prod_{j=1}^{P}\beta_{k_j}(t)\prod_{j=1}^{N}\overline{\beta_{l_j}(t)}|dt \leq |\prod_{j=1}^{P}c_{k_j}\prod_{j=1}^{N}c_{l_j}|\sigma^p,
	\]
	and summing again in $k_j$ and $l_j$ completes the proof. 
\end{proof}

\begin{lemma}
\label{technical.sum.of.S.lemma}
	For a fixed $I \in \N$, let $P_K := (I+K)/2$, $N_K := (I-K)/2$, and denote \mbox{$\phi_K = \sum_{j=1}^{P_K} \phi_{k_j} - \sum_{j=1}^{N_K} \phi_{l_j}$}, $\Gamma_K = \prod_{j=1}^{P_K}\gamma_{k_j}\prod_{j=1}^{N_K}\overline{\gamma_{l_j}}$, and $C_K = \prod_{j=1}^{P_K}c_{k_j}\prod_{j=1}^{N_K}\overline{c_{l_j}}$. Then,
	\begin{align*}
		&\sum_{K=0}^{I+1} \mathcal{S}_{I+1,K} = \sum_{K=1}^{I} \Big[\sum_{\substack{k_1, \ldots, k_{P_K+1} = 1 \\ l_1, \ldots, l_{N_K}}}^{\infty} C_K c_{k_{P_K+1}}  g_{I,K} e^{i(K+1)(\eta t + 2\theta)}e^{-i(\phi_K + \phi_{k_{P_K}+1})t} \Gamma_K\gamma_{P_K+1}  \\
		&\hspace{1.5in}- \sum_{\substack{k_1, \ldots, k_{P_K}=1 \\ l_1, \ldots, l_{N_K + 1} = 1}}^{\infty} C_K c_{l_{N_K}} g_{I,K} e^{i(K-1)(\eta t + 2\theta)} e^{-i(\phi_K - \phi_{l_{N_K}+1})t} \Gamma_K \overline{\gamma_{l_{N_K+1}}}\Big].
	\end{align*}
\end{lemma}
\begin{proof}
	Beginning from 
	\begin{align*}
		\sum_{K=0}^{I+1} \mathcal{S}_{I+1,K} &= \sum_{K=0}^{I+1} \sum_{\substack{k_1, \ldots, k_{P_K} = 1 \\ l_1, \ldots, l_{N_K} =1}}^\infty f_{I+1,K}(\eta; [\phi_{k_j}]_{j=1}^{P_K}; [\phi_{l_j}]_{j=1}^{N_K}) \prod_{j=1}^{P_K} \beta_{k_j}(x) \prod_{j=1}^{N_K} \overline{\beta_{l_j}(x)} e^{iK(\eta x + 2\theta(x))},
	\end{align*}
	we use \eqref{f.I.K} to rewrite as
	\begin{align*}
		&\sum_{K=0}^{I+1} \sum_{\substack{k_1, \ldots, k_{P_K} = 1 \\ l_1, \ldots, l_{N_K} =1}}^\infty \frac{1}{P_K!N_K!} \sum_{a=-1}^{1} \sum_{\substack{\sigma \in S_{P_K} \\ \tau \in S_{N_K}}} \omega_a g_{I, K + a}(\eta; [\phi_{k_{\sigma(j)}}]_{j=1}^{\min[P_K, P_K+ a]}; [\phi_{l_{\tau(j)}}]_{j=1}^{\min[N_K, N_K-a]}) \\
		&\hspace{2in}\times \prod_{j=1}^{P_K} \beta_{k_j}(x) \prod_{j=1}^{N_K} \overline{\beta_{l_j}(x)} e^{iK(\eta x + 2\theta(x))}.
	\end{align*}
	Grouping terms with $g_{I,\tilde{K}}$ of the same indices $(I,\tilde{K})$, each such group has two summands, one from the case where $K=\tilde{K}+1$ and $a = -1$ and another from the case where $K = \tilde{K}-1$ and $a = 1$. Each such term has the form
	\begin{align*}
		&\sum_{\substack{k_1, \ldots, k_{P_K+1} = 1 \\ l_1, \ldots, l_{N_K-1} = 1}}^{\infty} \frac{1}{(P_K + 1)!(N_K - 1)!} \sum_{\substack{\sigma \in S_{P_K + 1} \\ \tau \in S_{N_K - 1}}} g_{I,1}(\eta; [\phi_{k_{\sigma(j)}}]_{j=1}^{P_K}; [\phi_{l_{\tau(j)}}]_{j=1}^{N_K - 1}) \prod_{j=1}^{P_K + 1} \beta_{k_j} \prod_{j=1}^{N_K-1} \overline{\beta_{l_j}} \\ 
		&\hspace{.5in} - \sum_{\substack{k_1, \ldots, k_{P_{K}} = 1 \\ l_1, \ldots, l_{N_K} = 1}}^{\infty} \frac{1}{P_K!N_K!} \sum_{\substack{\sigma \in S_{P_K} \\ \tau \in S_{N_K}}} g_{I,1}(\eta; [\phi_{k_{\sigma(j)}}]_{j=1}^{P_K}; [\phi_{l_{\tau(j)}}]_{j=1}^{N_K - 1}) \prod_{j=1}^{P_K}\beta_{k_j} \prod_{j=1}^{N_K} \overline{\beta_{l_j}}. 
	\end{align*}
	By choosing to average over permutations of the frequencies in the first summand of the lemma, and summing in $K$ above, we have equality and complete the proof. 
\end{proof}

\begin{lemma}
	If $\varphi$ obeys the conditions of Lemma \ref{crux.lemma}, then for $I = 1, \ldots, p-1$ and $0 \leq a < b < \infty$, 
	\begin{equation}
		\Big|\int_{a}^{b} \big(\sum_{K=1}^{I}\mathcal{S}_{I,K}(t) - \sum_{K=0}^{I+1} \mathcal{S}_{I+1,K}(t) \big)dt\Big| \leq \sum_{K=1}^{I} \frac{1}{K}E_{I,K} \tau^I. \label{telescoping.lemma}
	\end{equation}
\end{lemma}
\begin{proof}
	For $K \geq 1$, we start from Lemma \ref{exchange.lemma} and multiply by $\frac{g_{I,K}}{iK}$to get
	\begin{align*}
		\big|\int_{a}^{b} f_{I,K} e^{iK(\eta t + 2\theta)}e^{-i\phi t} \Gamma + 2i g_{I,K} e^{iK(\eta t + 2\theta)} e^{-i\phi t} \Gamma \theta'\big| \leq \frac{2\tau^I}{K} |g_{I,K}|.
	\end{align*}
	Then we multiply through by $\prod_{j=1}^{P}c_{k_j}\prod_{j=1}^{N}\overline{c_{l_j}}$ and sum in all the $k_j$ and $l_j$ from one to infinity, which summation is justified by Fubini's theorem and Lemmas \ref{E.finite.lemma} and \ref{abs.conv.lemma}. Then summing in $K$ from one to $I$ finishes the proof. The term containing the $g_{I,K}$ becomes $\mathcal{S}_{I+1,K}$ using \eqref{theta.prime} and Lemma \ref{technical.sum.of.S.lemma}. 
\end{proof}

\begin{proof}[Proof of Lemma \ref{crux.lemma}]
	By summing over $I = 1, \ldots, p-1$ in \eqref{telescoping.lemma}, we have
	\begin{align*}
		\Big|\int_{a}^{b} \big(\mathcal{S}_{1,1}(t) - \sum_{K=1}^{p} \mathcal{S}_{p,K}(t) - \sum_{I=2}^{p} \mathcal{S}_{I,0}(t)\big)dt\Big| \leq \sum_{I=1}^{p-1}\sum_{K=1}^{I} \frac{1}{K} E_{I,K}\tau^I.
	\end{align*}
	Then, for $I=p$, we bound the sum in $K$ by
	\begin{align*}
		\big|\sum_{K=1}^{p} \int_{a}^{b} \mathcal{S}_{p,K}(t) dt\big| &\leq \sum_{K=1}^{p} \int_{a}^{b} \sum_{\substack{k_1, \ldots, k_{\frac{p+K}{2}}=1 \\ l_1, \ldots, l_{\frac{p-K}{2}} = 1}}^{\infty} |f_{p,K}\prod_{j=1}^{\frac{p+K}{2}} \beta_{k_j} \prod_{j=1}^{\frac{p-K}{2}} \overline{\beta_{l_{j}}} |dt \\
		&\leq \sum_{K=1}^{p} \sum_{a=-1}^{1} |\omega_a| E_{p-1, K+a} \sum_{j=1}^{\infty} |c_j| \sigma^p \leq 2\sum_{K=0}^{p-1} E_{p-1,K} \sum_{j=1}^\infty |c_j|\sigma^p. 
	\end{align*}
	The term we have yet to bound, $\sum_{I=2}^{p} \mathcal{S}_{I,0}(t)$, is independent of the Pr\"ufer variables, since $K = 0$. By condition \eqref{small.divisors.condition} of Lemma \ref{crux.lemma}, we can write the Pr\"ufer variables for two linearly independent solutions $U$ and $V$ as follows:
	\begin{align*}
		\log \frac{Z_U(x,\eta)}{Z_U(0,\eta)} &= \sum_{I=2}^{p} \int_{0}^{x}\mathcal{S}_{I,0}(t)dt + B_U(x), \\
		\log \frac{Z_V(x,\eta)}{Z_V(0,\eta)} &= \sum_{I=2}^{p} \int_{0}^{x}\mathcal{S}_{I,0}(t)dt + B_V(x),
	\end{align*}
	where the terms $B_U$ and $B_V$ are bounded functions that depend on $U$ and $V$, respectively, whereas $\sum_{I=2}^{p} \int_{0}^{x}\mathcal{S}_{I,0}(t)dt$ does not depend on the solution $U$ or $V$. Thus the difference is given by 
	\[
		\log \frac{Z_U(x,\eta)}{Z_V(x,\eta)} - \log \frac{Z_U(0,\eta)}{Z_V(0,\eta)} = B_U(x) - B_V(x).
	\]
	Since $B_U - B_V$ has a finite limit at infinity, taking the real part shows that $\log \frac{r_U(x,\eta)}{r_V(x,\eta)}$ also has a finite limit at infinity, and thus $\frac{r_U(x,\eta)}{r_V(x,\eta)}$ has a finite, nonzero limit at infinity. This holds independently of the choice of $U$ and $V$.  

	Note that the difference $B_U - B_V$ is absolutely continuous, since it may be written as 
	\[
		B_U(x) - B_V(x) = \log \frac{Z_U(x,\eta)}{Z_U(0,\eta)} - \log \frac{Z_V(x,\eta)}{Z_V(0,\eta)} = \int_{0}^{x} e^{i\eta t} \varphi(t) (e^{2i\theta_U(t,\eta)} - e^{2i\theta_V(t,\eta)}) dt.
	\]
	Since $B_U - B_V$ has a finite limit at infinity, there exists some $0 < x_0 < \infty$ for which 
	\[
		|\int_{x_0}^{\infty} (B_U(t) - B_V(t))' dt | < \frac{\pi}{4}.
	\]
	Taking the imaginary part then gives that $|\theta_U(x,\eta) - \theta_V(x,\eta) - (\theta_U(0,\eta) - \theta_V(0,\eta))|$ is bounded by $\pi/4$ for $x \geq x_0$. In particular, if we choose the solution $U$ arbitrarily and then choose $V$ such that $\theta_V(0,\eta) = \theta_U(0,\eta) - \pi/2 + 2k\pi$ for some $k \in \Z$, then for $x \geq x_0$, we have
	\[
		\theta_U(x,\eta) - \theta_V(x,\eta) \in (\frac{\pi}{4}, \frac{3\pi}{4}). 
	\]
	In particular, $\sin(\theta_U(x,\eta) - \theta_V(x,\eta)) \in (\frac{\sqrt{2}}{2}, 1]$. Since $U$ and $V$ are linearly independent, their Wronskian is independent of $x$ and is given by some nonzero constant $C$. We write
	\[
		W[U,V](x) = 4 r_U(x,\eta)r_V(x,\eta)\sin(\theta_U(x,\eta) - \theta_V(x,\eta)) = C \neq 0.
	\]
	Multiplying through by $\frac{r_U(x,\eta)}{r_V(x,\eta)}$, the boundedness of $\sin(\theta_U(x,\eta) - \theta_V(x,\eta))$ away from zero implies $r_U^2(x,\eta)$, and therefore $r_U(x,\eta)$, has a finite limit at infinity. Since $U,V$ were chosen arbitrarily, every solution of \eqref{eigenequation} is bounded. 
\end{proof}

\begin{remark}
It is quicker to merely show absence of subordinate solutions, rather than to show boundedness of solutions. If $U$ were a subordinate solution, then for any linearly independent solution $V$ we would have, by L'H\^opital's rule, 
	\begin{align*}
		0 &= \xlim \frac{\int_{0}^{x}\|U(t,\eta)\|^2dt}{\int_{0}^{x}\|V(t,\eta)\|^2dt} = \xlim \frac{\|U(x,\eta)\|^2}{\|V(x,\eta)\|^2} = \xlim \frac{|r_U(x,\eta)|^2}{|r_V(x,\eta)|^2},
	\end{align*}
which contradicts the fact from the first part of the proof that $\frac{r_U(x,\eta)}{r_V(x,\eta)} \to L$ for some nonzero $L$.
\end{remark}

For the following lemma, recall that the Catalan numbers $C_n$ are given by $C_n = \frac{1}{n+1}\binom{2n}{n}$. 

\begin{lemma}
\label{UBH.lemma}
Let $\nu$ be a finite uniformly $\beta$-H\"older measure on $\R$. 
\begin{enumerate}
	\item If $\alpha \in (0,\beta)$, then for all $\psi \in \R$, 
	\begin{equation}
		\int \frac{1}{|\psi - \eta|^\alpha}d\nu(\eta) \leq D_{\alpha}, \label{beta.holder.measure.part.i}
	\end{equation}
	where $D_\alpha$ is a finite constant that depends only on $\alpha$. 
	\item For $I\geq 1$, $I=2P-1$, and $\alpha \in (0,\frac\beta I)$, 
	\begin{equation}
		\int |h_I(\eta; [\phi_{k_j}]_{j=1}^P; [\phi_{l_j}]_{j=1}^{P-1}) |^\alpha d\nu(\eta) \leq C_I D_{I\alpha}, \label{beta.holder.measure.part.ii}
	\end{equation} 
	where $C_I$ are Catalan numbers.
\end{enumerate}
\end{lemma}
\begin{proof} \
\begin{enumerate}
	\item This is proved in ~\cite[Lemma 4.1]{L2014}.
	\item We induct on $P$. The case $P=1$ holds by \eqref{beta.holder.measure.part.i}. Suppose the statement holds up to $P-1$. Integrating one summand in \eqref{h.I} and using H\"older's inequality and the induction hypothesis gives
	\begin{align*}
		\int \Big|\frac{1}{\sum_{j=1}^{P}\phi_{k_j}-\sum_{j=1}^{P-1}\phi_{l_j} - \eta} h_m h_{I-1-m}\Big|^\alpha d\nu(\eta) &\leq D_{I\alpha}^{1/I} (C_mD_{I\alpha})^{m/I} (C_{I-1-m}D_{I\alpha})^{(I-1-m)/I} \\
		&\leq C_mC_{I-1-m} D_{I\alpha},
	\end{align*}
	and since the Catalan sequence obeys the recursion $C_n = \sum_{j=0}^{n-1} C_jC_{n-1-j}$, summing in $0 \leq m \leq I-1$ recovers the Catalan number $C_I$.  
\end{enumerate}
\end{proof}

\begin{lemma}
\label{Hdim.bound.lemma}
Suppose \eqref{alpha.type.decay.condition} holds and let $I = 2P-1$. Then the set of $\eta$ for which 
\[
	\sum_{\substack{k_1, \ldots, k_P =1 \\ l_1, \ldots, l_{P-1} = 1}}^\infty \Big|h_I(\eta; [\phi_{k_j}]_{j=1}^P; [\phi_{l_j}]_{j=1}^{P-1})\prod_{j=1}^{P}c_{k_j} \prod_{j=1}^{N} \overline{c_{l_j}}\Big|\tau^I = \infty, 
\]
has Hausdorff dimension at most $I\alpha$. If $I = 2P$, the same set is empty.
\end{lemma}
\begin{proof}
Let $T$ be the set of $\eta$ for which condition \eqref{small.divisors.condition} of Lemma \ref{crux.lemma} fails. Suppose the Hausdorff dimension of $T$ is greater than $i\alpha$. Then for some $\beta > i\alpha$, $h^\beta(T)=\infty$. This implies the existence of a subset $T'\subset T$ such that $\nu = \chi_{T'}h^\beta$ is a finite uniformly $\beta$-H\"older measure with $\nu(T)>0$. Then Lemma \ref{UBH.lemma} implies 
\begin{align*}
	&\int \sum_{\substack{k_1, \ldots, k_{\frac{i+1}{2}}=1 \\ l_1, \ldots, l_{\frac{i-1}{2}} =1}}^\infty |h_i(\eta; [\phi_{k_j}]_{j=1}^{\frac{i+1}{2}}; [\phi_{l_j}]_{j=1}^{\frac{i-1}{2}})\prod_{j=1}^{\frac{i+1}{2}} c_{k_j} \prod_{j=1}^{\frac{i-1}{2}} \overline{c_{l_j}}|^\alpha d\nu(\eta) \\
	&\hspace{1in}= \sum_{\substack{k_1, \ldots, k_{\frac{i+1}{2}}=1 \\ l_1, \ldots, l_{\frac{i-1}{2}}=1}}^\infty |\prod_{j=1}^{\frac{i+1}{2}} c_{k_j} \prod_{j=1}^{\frac{i-1}{2}} c_{l_j}|^\alpha \int |h_i(\eta; [\phi_{k_j}]_{j=1}^{\frac{i+1}{2}}; [\phi_{l_j}]_{j=1}^{\frac{i-1}{2}})|^\alpha d\nu(\eta) \\
	&\hspace{1in}\leq \sum_{\substack{k_1, \ldots, k_{\frac{i+1}{2}}=1 \\ l_1, \ldots, l_{\frac{i-1}{2}}=1}}^\infty |\prod_{j=1}^{\frac{i+1}{2}} c_{k_j} \prod_{j=1}^{\frac{i-1}{2}} c_{l_j}|^\alpha D_{i\alpha} = D_{i\alpha}\Big(\sum_{j =1}^\infty |c_j|^\alpha\Big)^i,
\end{align*}
which is finite by the $\alpha$-type decay condition \eqref{alpha.type.decay.condition}. Since the integral is finite, the integrand is $\nu$-almost everywhere finite. But since $\alpha \in (0,1]$, this implies that condition \eqref{small.divisors.condition} holds $\nu$-almost everywhere, so that $\nu(T) = 0$, a contradiction. The second statement of the lemma follows from the fact that $h_I$ is zero for even $I$. 
\end{proof}

\begin{proof}[Proof of Theorem \ref{infinite.main.thm}]
Conditions \eqref{uniformly.b.v.condition}, \eqref{uniformly.L.p.condition}, and \eqref{alpha.type.decay.condition} of Lemma \ref{crux.lemma} are trivially satisfied for every $\eta$. Condition \eqref{small.divisors.condition} is satisfied away from a set $T$ of Hausdorff dimension at most $(p-2)\alpha$ by Lemma \ref{Hdim.bound.lemma} (note that condition \eqref{small.divisors.condition} of Lemma \ref{crux.lemma} need only be satisfied for odd $I =1, \ldots, p-2$). 

By Lemma \ref{crux.lemma} and Lemma \ref{stolz.type.lemma}, there are no subordinate solutions for $2E = \eta \in \R \setminus S$, and by the subordinacy theory of Gilbert-Pearson and Behncke, this implies the theorem. 
\end{proof}
\end{section}

\bibliographystyle{amsplain}

\begin{thebibliography}{10}
	
	\bibitem{A54}
	F.V.~Atkinson, \emph{The asymptotic solution of second-order differential equations}, Annali di Matematica Pura ed Applicata \textbf{37} (1954), no.~1, 347--378. 

	\bibitem{A2009}
	A.~Avila, \emph{On the spectrum and Lyapunov exponent of limit periodic Schr{\"o}dinger operators}, Comm. in Math. Phys. \textbf{288} (2009), no.~3, 907--918.

	\bibitem{B91}
	H.~Behncke, \emph{Absolute continuity of {H}amiltonians with von {N}eumann
              {W}igner potentials. {II}}, Manuscripta Math. \textbf{71} (1991), no.~2, 163--181.

    \bibitem{BD79}
    M.~Ben-Artzi and A.~Devinatz, \emph{Spectral and scattering theory for the adiabatic oscillator and related potentials}, Journal of Math. Phys. \textbf{20} (1979), no.~4, 594--607.

    \bibitem{BR95}
    H.~Behncke and P.~Rejto, \emph{A limiting absorption principle for separated Dirac operators with Wigner Von Neumann type potentials}, Hamiltonian Dynamical Systems (1995), 59--88.

	\bibitem{CG2002} S.~Clark and F.~Gesztesy,
     \emph{Weyl-{T}itchmarsh {$M$}-function asymptotics, local uniqueness
              results, trace formulas, and {B}org-type theorems for {D}irac
              operators}, Trans. Amer. Math. Soc. \textbf{354} (2002), no.~9, 3475--3534.

    \bibitem{EGL2020}
    B.~Eichinger, E.~Gwaltney, and M.~Luki\'c
    \emph{Stahl-Totik Regularity for Dirac Operators}, arXiv:2012.12889 (2020)

    \bibitem{FL21}
    R.L.~Frank and S.~Larson, \emph{Discrete Schr\"odinger operators with decaying and oscillating potentials}, arXiv:2108.05083 (2021).

    \bibitem{GP87}
    D.J.~Gilbert and D.B.~Pearson, \emph{On subordinacy and analysis of the spectrum of one-dimensional {S}chr\"{o}dinger operators}, J. Math. Anal. Appl. \textbf{128} (1987), no.~1, 30--56.


	\bibitem{GK2014} B.~Gr\'ebert, T.~Kappeler,  \emph{The defocusing NLS equation and its normal form}.
	EMS Series of Lectures in Mathematics. European Mathematical Society (EMS), Z\"urich, 2014. x+166 pp. 

	\bibitem{HL75}
	W.A.~Harris, Jr. and D.A.~Lutz, \emph{Asymptotic integration of adiabatic oscillators}, J. Math. Anal. Appl. \textbf{51} (1975), no.~1, 76--93.

	\bibitem{JS10}
	J.~Janas and S.~Simonov, \emph{Weyl-Titchmarsh type formula for discrete Schr{\"o}dinger operator with Wigner-von Neumann potential}, arXiv:1003.3319 (2010).

	\bibitem{K2005}
	A.~Kiselev, \emph{Imbedded singular continuous spectrum for Schr{\"o}dinger operators}, J. Amer. Math. Soc. \textbf{18} (2005), no.~3, 571--603.

	\bibitem{KLS98}
	A.~Kiselev, Y.~Last, and B.~Simon, \emph{Modified {P}r\"{u}fer and {EFGP} transforms and the spectral analysis of one-dimensional {S}chr\"{o}dinger operators}, Comm. Math. Phys. \textbf{194} (1998), no.~1, 1--45.

	\bibitem{K2012}
	H.~Kr\"uger, \emph{On the existence of embedded eigenvalues}, J. Math. Anal. Appl. \textbf{395} (2012), no.~2, 776--787.

	\bibitem{L2011}
	M.~Luki\'c, \emph{Orthogonal polynomials with recursion coefficients of generalized bounded variation}, Comm. Math. Phys. \textbf{306} (2011), no.~2, 485--509.

	\bibitem{L2013DiscreteGBV}
	M.~Luki\'c, \emph{Jacobi and {CMV} matrices with coefficients of generalized bounded variation}, Operator methods in mathematical physics, Oper. Theory Adv. Appl. \textbf{227} (2013), 117--121. 

	\bibitem{L2013SlowWvN}
	M.~Luki\'c, \emph{Schr\"{o}dinger operators with slowly decaying {W}igner-von {N}eumann type potentials}, J. Spectr. Theory \textbf{3} (2013), no.~2, 147--169.

	\bibitem{L2014}
	M.~Luki\'c, \emph{A class of {S}chr\"{o}dinger operators with decaying oscillatory potentials}, Comm. Math. Phys. \textbf{326} (2014), no.~2, 441--458.

	\bibitem{LO2015}
	M.~Luki\'c and D.C. Ong, \emph{Wigner-von {N}eumann type perturbations of periodic {S}chr\"{o}dinger operators}, Tran. Amer. Math. Soc. \textbf{367} (2105), no.~1, 707--724.

	\bibitem{LO16}
	M.~Luki\'c and D.C. Ong, \emph{Generalized Pr{\"u}fer variables for perturbations of Jacobi and CMV matrices}, J. Math. Anal. Appl. \textbf{444} (2016), no.~2, 1490--1514.

	\bibitem{LS2017}
	J.~L\H{o}rinczi and I.~Sasaki, \emph{Embedded eigenvalues and Neumann--Wigner potentials for relativistic Schr{\"o}dinger operators}, J. Funct. Anal. \textbf{273} (2017), no.~4, 1548--1575.

	\bibitem{N86}
	S.N.~Naboko, \emph{Dense point spectra of Schrodinger and Dirac operators}, Teoreticheskaya i Matematicheskaya Fizika \textbf{68} (1986), no.~1, 18--28.

	\bibitem{NS12}
	S.N.~Naboko and S.~Simonov, \emph{Zeroes of the spectral density of the periodic Schr{\"o}dinger operator with Wigner--von Neumann potential}, Mathematical Proceedings of the Cambridge Philosophical Society \textbf{153} (2012), no.~1, 33--58.

	\bibitem{vNW29}
	J.~von Neumann and E.~Wigner, \emph{On some peculiar discrete eigenvalues}, Phys. Z. \textbf{465} (1929). 

	\bibitem{P26}
	H.~Pr\"ufer, \emph{Neue {H}erleitung der {S}turm-{L}iouvilleschen {R}eihenentwicklung stetiger {F}unktionen}, Math. Ann. \textbf{95} (1926), no.~1, 499--518.

	\bibitem{S2016}
	K.M.~Schmidt, \emph{On the Asymptotics of the Spectral Density of Radial Dirac Operators with Divergent Potential}, Integral Equations and Operator Theory \textbf{85} (2016), no.~1, 137--149. 

	\bibitem{S96}
	B.~Simon, \emph{Bounded eigenfunctions and absolutely continuous spectra for one-dimensional Schr{\"o}dinger operators}, Proc. Amer. Math. Soc. \textbf{124} (1996), no.~11, 3361--3369. 

	\bibitem{S97}
	B.~Simon, \emph{Some Schr{\"o}dinger operators with dense point spectrum}, Proc. Amer. Math. Soc. \textbf{125} (1997), no.~1, 203--208.

	\bibitem{S12}
	S.~Simonov, \emph{Zeroes of the spectral density of discrete Schr{\"o}dinger operator with Wigner-von Neumann potential}, Integral Equations and Operator Theory \textbf{73} (2012), no.~3, 351--364.

	\bibitem{Sim2016}
	S.~Simonov, \emph{Zeroes of the spectral density of the Schr{\"o}dinger operator with the slowly decaying Wigner--von Neumann potential}, Mathematische Zeitschrift \textbf{284} (2016), no.~1, 335--411.

	\bibitem{ST2010}
	R.~Stadler and G.~Teschl, \emph{Relative oscillation theory for Dirac operators}, J. Math. Anal. Appl. \textbf{371} (2010), no.~2, 638--648. 

	\bibitem{S92}
	G.~Stolz, \emph{Bounded solutions and absolute continuity of Sturm-Liouville operators}, J. Math. Anal. Appl. \textbf{169} (1992), no.~1, 210--228.

	\bibitem{W67}
	J.~Weidmann, \emph{Zur Spektraltheorie von Sturm-Liouville-Operatoren}, Mathematische Zeitschrift \textbf{98} (1967), no.~4, 268--302.

	\bibitem{W83}
	D.A.~White, \emph{Schr{\"o}dinger operators with rapidly oscillating central potentials}, Trans. Amer. Math. Soc. \textbf{275} (1983), no.~2, 641--677. 

	\bibitem{W09}
	M.L.~Wong, \emph{Generalized bounded variation and inserting point masses}, Constructive Approximation \textbf{30} (2009), no.~1, 1--15. 

\end{thebibliography}

\providecommand{\MR}[1]{}
\providecommand{\bysame}{\leavevmode\hbox to3em{\hrulefill}\thinspace}
\providecommand{\MR}{\relax\ifhmode\unskip\space\fi MR }
\providecommand{\MRhref}[2]{%
	\href{http://www.ams.org/mathscinet-getitem?mr=#1}{#2}
}
\providecommand{\href}[2]{#2}

\end{document}